\documentclass[a4paper,11pt,reqno]{article}

\usepackage[utf8]{inputenc}
\usepackage[T1]{fontenc}
\usepackage{lmodern}
\usepackage[english]{babel}
\usepackage{microtype}

\usepackage{amsmath,amssymb,amsfonts,amsthm}
\usepackage{mathtools,accents}
\usepackage{mathrsfs}
\usepackage{aliascnt}
\usepackage{braket}
\usepackage{bm}

\usepackage[a4paper,margin=3cm]{geometry}
\usepackage[citecolor=blue,colorlinks]{hyperref}

\usepackage{enumerate}
\usepackage{xcolor}

\makeatletter
\g@addto@macro\@floatboxreset\centering
\makeatother

\makeatletter
\def\newaliasedtheorem#1[#2]#3{
  \newaliascnt{#1@alt}{#2}
  \newtheorem{#1}[#1@alt]{#3}
  \expandafter\newcommand\csname #1@altname\endcsname{#3}
}
\makeatother

\numberwithin{equation}{section}

\newtheoremstyle{slanted}{\topsep}{\topsep}{\slshape}{}{\bfseries}{.}{.5em}{}

\theoremstyle{plain}
\newtheorem{theorem}{Theorem}[section]
\newaliasedtheorem{proposition}[theorem]{Proposition}
\newaliasedtheorem{lemma}[theorem]{Lemma}
\newaliasedtheorem{corollary}[theorem]{Corollary}
\newaliasedtheorem{conjecture}[theorem]{Conjecture}
\newaliasedtheorem{counterexample}[theorem]{Counterexample}

\theoremstyle{definition}
\newaliasedtheorem{definition}[theorem]{Definition}
\newaliasedtheorem{question}[theorem]{Question}
\newaliasedtheorem{example}[theorem]{Example}

\theoremstyle{remark}
\newaliasedtheorem{remark}[theorem]{Remark}

\newcommand{\setR}{\mathbb{R}}

\newcommand{\eps}{\varepsilon}

\let\altphi\phi
\let\phi\varphi
\let\varphi\altphi
\let\altphi\undefined

\newcommand{\di}{\mathop{}\!\mathrm{d}}

\DeclareMathOperator{\Lip}{Lip}

\newcommand{\dist}{\mathsf{d}}

\newcommand{\meas}{\mathfrak{m}}

\DeclareMathOperator{\RCD}{RCD}
\DeclareMathOperator{\CD}{CD}

\newfont{\tmpf}{cmsy10 scaled 2500}

\newcommand{\m}{\mathfrak{m}}
\renewcommand{\H}{\mathcal{H}}
\newcommand{\intav}{{\mathop{\int\kern-10pt\rotatebox{0}{\textbf{--}}}}}
\renewcommand{\d}{\mathrm{d}}
\renewcommand{\ }{\text{ }}
\def\<{\langle}
\def\>{\rangle}
\newcommand{\ra}{\rightarrow}
\newcommand{\R}{\setR}
\def\Sum#1#2{\mathop{\sum}\limits_{#1}^{#2}}
\def\Prod#1#2{\mathop{\prod}\limits_{#1}^{#2}}

\def\esssup#1{\mathop{\mathrm{ess\text{ }sup}}\limits_{#1}}
\def\essinf#1{\mathop{\mathrm{ess\text{ }inf}}\limits_{#1}}
\newcommand{\Test}{\mathrm{Test}}

\newcommand{\rad}{\mathop{\mathrm{rad}}}
\newcommand{\diam}{\mathop{\mathrm{diam}}}
\def\Sup#1{\mathop{\sup}\limits_{#1}}
\def\Inf#1{\mathop{\inf}\limits_{#1}}
\newcommand{\leqae}{\mathop{\leqslant}\limits^{\text{a.e.}}}
\newcommand{\Hess}{{\textrm{Hess}}}

\begin{document}
\title{A note on the topological stability theorem from $\RCD$ spaces to Riemannian manifolds}
\author{
Shouhei Honda,
\thanks{Mathematical Institute, Tohoku University, \url{shouhei.honda.e4@tohoku.ac.jp}} \and 
Yuanlin Peng,
\thanks{Mathematical Institute, Tohoku University, \url{peng.yuanlin.p6@dc.tohoku.ac.jp }} } 
\maketitle
\begin{abstract}
Inspired by a recent work of Wang-Zhao in \cite{WangZhao}, in this note we prove that for a fixed $n$-dimensional closed Riemannian manifold $(M^n, g)$, if an $\RCD(K, n)$ space $(X, \dist, \meas)$ is Gromov-Hausdorff close to $M^n$, then there exists a regular homeomorphism $F$ from $X$ to $M^n$ such that $F$ is Lipschitz continuous and that $F^{-1}$ is H\"older continuous, where the Lipschitz constant of $F$, the H\"older exponent and the H\"older constant of $F^{-1}$ can be chosen arbitrary close to $1$. This is sharp in the sense that in general such a map cannot be improved to being bi-Lipschitz. Moreover if $X$ is smooth, then such a homeomorphism can be chosen as a diffeomorphism. It is worth mentioning that the Lipschitz-H\"older continuity of $F$ improves the intrinsic Reifenberg theorem for closed manifolds with Ricci curvature bounded below established by Cheeger-Colding. The Nash embedding theorem plays a key role in the proof.
\end{abstract}

\tableofcontents

\section{Introduction}
\subsection{Topological stability theorem}
Let us consider a Gromov-Hausdorff convergent sequence of compact metric spaces:
\begin{equation}\label{eq:ghxonv}
(X_i, \dist_i) \stackrel{\mathrm{GH}}{\to} (X, \dist).
\end{equation}
A fundamental problem in this setting is to find \textit{nice} (Gromov-Hausdorff approximating) maps $F_i: X_i \to X$, under a suitable restriction of curvature of $X_i$ in a synthetic sense. For example if each $X_i$ is a finite dimensional Alexandrov space with curvature bounded below by $-1$ and the dimensions of $X_i$ and of $X$ coincide, then for any sufficiently large $i$, $F_i$ can be chosen as a homeomorphism, which is proved by Perelman in \cite{Pere} (see also \cite{Kapo}). This result is called a \textit{topological stability theorem}, related to sectional curvature.

Next let us consider the case of Ricci curvature. Then  instead of (\ref{eq:ghxonv}), it is better to consider the following setting:
\begin{equation}\label{eq:mghconv2}
\left(X_i, \dist_i, \meas_i\right) \stackrel{\mathrm{mGH}}{\to} (X, \dist, \meas),
\end{equation}
where $\meas_i, \meas$ are Borel measures on compact spaces $X_i, X$, respectively and the convergence in (\ref{eq:mghconv2}) stands for measured Gromov-Hausdorff convergence. 

For the sequence (\ref{eq:mghconv2}) assume that all $(X_i, \dist_i, \meas_i)$ are $\RCD(K, N)$ spaces for fixed $K \in \mathbb{R}$ and $N \in [1, \infty)$, namely they satisfy:
\begin{itemize}
\item the Ricci curvature are bounded below by $K$, the dimensions are bounded above by $N$ in a synthetic sense, and the $H^{1,2}$-Sobolev spaces are Hilbert.
\end{itemize}
See \cite{A} for a survey on this topic.
Then it is natural to ask whether a similar topological stability theorem as in the case of sectional curvature is satisfied in this setting, namely 
\begin{enumerate}
\item[(Q1)] If the dimensions of $X_i$ and $X$ coincide in  (\ref{eq:mghconv2}), then does there exist a homeomorphism $F_i:X_i \to X$? Here, the dimension stands for the essential dimension proved in \cite{BrueSemola} (after \cite{CN} for Ricci lmit spaces).
\end{enumerate}
Note that the essential dimension is not necessary to coincide with the Hausdorff dimension by \cite{PW} and that the question (Q1) has a positive answer if $N=1, 2$ or if $N=3$ and $X_i$ is smooth, see \cite{LW, ST}.
It is well-known that in general the question (Q1) has a negative answer even if each $(X_i, \dist_i)$ is isometric to a Ricci flat manifold and the essential dimensions are equal to $N=4$ (see for example \cite{KT, Page}). Let us emphasize that the reason is in the existence of singular points of $X$.

Actually the intrinsic Reifenberg theorem proved in \cite{CheegerColding1} with (almost) rigidity results obtained in \cite{DePhillippisGigli} allows us to prove that if $X$ has no singular points and the essential dimension is equal to $N$, then there exists a homeomorphism $F_i:X_i \to X$ for any sufficiently large $i$. Moreover such $F_i$ can be chosen as bi-H\"older homeomorphisms, where the H\"older exponent can be chosen arbitrary close to $1$. See \cite{KM}.

Thus it is also natural to ask;
\begin{enumerate}
\item[(Q2)] If $X$ has no singular points and the essential dimension is equal to $N$, then does there exist \textit{canonical} homeomorphism $F_i:X_i \to X$?
\end{enumerate}
The first positive answer to this question (Q2) is recently given in \cite{WangZhao} who proved that if  $X$ is isometric to the stadard unit $N$-sphere $(\mathbb{S}^N, \dist_{\mathbb{S}^N})$ and each $X_i$ is smooth with positive Ricci curvature, then for any $\epsilon \in (0, 1)$ such $F_i$ can be chosen as normalized eigenmap satisfying that $F_i$ is a $(1+\epsilon)$-Lipschitz diffeomorphism and $F_i^{-1}$ is $(1-\epsilon)$-H\"older continuous with the H\"older constant at most $(1+\epsilon)$.

The main purpose of this note is to generalize the result above in \cite{WangZhao} to the case when $(X, \dist)$ is smooth, not necessarily isometric to a sphere and $X_i$ is not necessarily smooth. In our setting, instead of using eigenmaps, we adopt fixing a smooth isometric embedding $\Phi:X \hookrightarrow \mathbb{R}^k$ whose general existence is guaranteed by the Nash embedding theorem (note that in the case of the unit sphere, the canonical inclusion $\iota:\mathbb{S}^N \hookrightarrow \mathbb{R}^{N+1}$ is an isometric eigenmap). In particular as a corollary of the main results, Theorem \ref{thm:canonicaldiff}, we get the following result;
\begin{theorem}
Let $N \in \mathbb{N}$ and let $(M^N, g)$ be an $N$-dimensional closed Riemannian manifold. For any $K \in \mathbb{R}$ and any $\epsilon \in (0, 1)$ there exists $\delta \in (0, 1)$ such that if a compact $\RCD(K, N)$ space $(X, \dist, \meas)$ is $\delta$-Gromov-Hausdorff close to $(M^N, \dist_g)$, then there exists a homeomorphism $F$ from  $X$ to $M^N$ such that 
\begin{equation}\label{eq:biholderlip44}
(1-\epsilon)\dist(x, y)^{1+\epsilon}\le \dist_g(F(x), F(y))\le (1+\epsilon)\dist (x, y),\quad \forall x, y \in X.
\end{equation}
\end{theorem}

Let us emphasize that the developments along this direction, including the theorem above, are heavily based on recent techniques established in \cite{CheegerNaber, CheegerJiangNaber}.

\subsection{Results}
We say that a map $F:U \to M^k$ from a bounded open subset $U$ of an $\RCD(K, N)$ space $(X, \dist, \meas)$ to a $k$-dimensional (not necessarily complete) Riemannian manifold $(M^k, g)$ is \textit{regular} if $\phi \circ F \in D(\Delta, U)$ with $\Delta (\phi \circ F) \in L^{\infty}(U, \meas)$ for any $\phi \in C^{\infty}(M^k)$. Since any smooth manifold can be smoothly embedded into a Euclidean space, we know that any such a regular map $F$ is locally Lipschitz (see Proposition \ref{prop:regulargradient}). 

The main results of the paper are stated as follows. 
\begin{theorem}\label{thm:canonicaldiff}
Let us consider a measured Gromov-Hausdorff convergent sequence of compact $\RCD(K, N)$ spaces:
\begin{equation}\label{eq:mghconv}
(X_i, \dist_i, \meas_i) \stackrel{\mathrm{mGH}}{\to} (X, \dist, \meas).
\end{equation}
Assume that $N$ is an integer and $(X, \dist)$ is isometric to an $N$-dimensional closed Riemannian manifold $(M^N, \dist_g)$. Then  we have the following.
\begin{enumerate}
\item\label{1} If a sequence of regular maps $F_i:X_i \to M^N$ is  equi-regular, namely,
\begin{equation}
\sup_i\|\Delta_i (\phi \circ F_i)\|_{L^{\infty}}<\infty, \quad \forall \phi \in C^{\infty}(M^N),
\end{equation}
and satisfies
\begin{equation}\label{eq:pull}
\int_{X_i}\left| g_i-F^*_ig\right|\di \meas_i \to 0,
\end{equation}
where $\Delta_i, g_i$ denote the Laplacian,  the canonical Riemannian metric of $(X_i, \dist_i, \meas_i)$, respectively, then we have the following.
\begin{enumerate}
\item For any sufficiently large $i$, $F_i$ gives a homeomorphism and an $\epsilon_i$-Gromov-Hausdorff approximation with
\begin{equation}\label{isbbs}
(1-\epsilon_i)\dist_i (x_i, y_i)^{1+\epsilon_i}\le \dist_g(F_i(x_i), F_i(y_i)) \le C(N)\dist_i(x_i, y_i),\quad \forall x_i, y_i \in X_i
\end{equation}
for some $\epsilon_i \to 0^+$ and  some positive constant $C(N)$ depending only on $N$.
\item If each $(X_i, \dist_i)$ is isometric to an $N$-dimensional closed Riemannian manifold $(M^N_i, \dist_{g_i})$ and $F_i$ is smooth, then $F_i$ gives a diffeomorphism for any sufficiently large $i$.  
\end{enumerate}
\item\label{2} Let $\Phi_i=(\phi_{i, 1}, \ldots, \phi_{i, k}):X_i \to \mathbb{R}^k$ be a sequence of equi-regular maps converging uniformly to a smooth isometric  (namely, $\Phi^*g_{\mathbb{R}^k}=g$) embedding $\Phi:M^N \hookrightarrow \mathbb{R}^k$ and let $\pi$ be a $C^{1, 1}$-Riemannian submersion from a neighborhood of $\Phi(M^N)$ to $\Phi(M^N)$ with $\pi|_{\Phi(M^N)}=\mathrm{id}_{\Phi(M^N)}$. If  the sequence $\Delta_i\phi_{i, j}$ is equi-Lipschitz,
then the sequence $F_i:=\Phi^{-1}\circ \pi \circ \Phi_i:X_i \to M^N$ is equi-regular with (\ref{eq:pull}), and (\ref{isbbs}) can be improved to
\begin{equation}\label{eq:bihl3}
(1-\epsilon_i)\dist_i (x_i, y_i)^{1+\epsilon_i}\le \dist_g(F_i(x_i), F_i(y_i)) \le (1+\epsilon_i)  \dist_i(x_i, y_i),\quad \forall x_i, y_i \in X_i
\end{equation}
for some $\epsilon_i \to 0^+$.
\item\label{3} Let $\Phi=(\phi_1, \ldots, \phi_k):M^N \hookrightarrow \mathbb{R}^k$ be a smooth isometric embedding. 
Then we have the following.
\begin{enumerate}
\item\label{a} There exists a sequence of equi-regular maps $\Phi_i=(\phi_{i, 1},\ldots, \phi_{i, k}):X_i \to \mathbb{R}^k$ such that $\Phi_i$ converge uniformly to $\Phi$ and that the sequence $\Delta_i \phi_{i, j}$ is equi-Lipschitz and $H^{1,2}$-strongly converge to $\Delta \phi_j$.   
\item\label{b} If each $(X_i, \dist_i)$ is isometric to an $N$-dimensional closed Riemannian manifold $(M^N_i, \dist_{g_i})$, then such $\Phi_i$ as in (\ref{a}) can be chosen as smooth maps. 
\end{enumerate}
\end{enumerate}
\end{theorem}
Let us give few comments on Theorem \ref{thm:canonicaldiff}. Under the same setting as above, 
\begin{itemize}
\item there exists a convergent sequence of positive numbers $c_i \to c$ in $(0, \infty)$ such that $\meas=c\mathcal{H}^N$ and $\meas_i=c_i\mathcal{H}^N$ hold for any sufficiently large $i$. See \cite{KM} (see also \cite{BGHZ, H}). Namely, after normalizations of the reference measures, (\ref{eq:mghconv}) is a non-collapsed sequence of non-collapsed compact $\RCD(K, N)$ spaces (see section \ref{sec2});
\item such $F_i^{-1}$ as in (\ref{1}) is not  locally Lipschitz for any sufficiently large $i$, whenever $X$ has a singular point. This is a direct consequence of a characterization of the $N$-dimensional Euclidean space obtained in \cite{HS} (see Proposition \ref{propbilip});
\item for a sequence of regular maps $F_i:X_i \to M^N$, it is equi-regular if and only if for any $j$,
\begin{equation}
\sup_i\|\Delta_i(\phi_j \circ F_i)\|_{L^{\infty}}<\infty, 
\end{equation}
where $\Phi=(\phi_1, \ldots, \phi_k):M^N \hookrightarrow \mathbb{R}^k$ is a smooth isometric embedding, because for all $\RCD(K, \infty)$ space $(X, \dist, \meas)$, Lipschitz functions $f_i \in D(\Delta) (i=1, 2,\ldots k)$ and $C^{1,1}$-Lipschitz function $\phi:\mathbb{R}^k \to \mathbb{R}$, letting $F=(f_1,\ldots, f_k)$ we have $\phi \circ F \in D(\Delta)$ with 
\begin{equation}\label{eq:leipnit}
\Delta (\phi \circ F)=\sum_{i, j=1}^k\frac{\partial^2\phi}{\partial x_i \partial x_j}(F(x))\langle \nabla f_i, \nabla f_j\rangle (x)+\sum_{i=1}^k\frac{\partial \phi}{\partial x_i}(F(x))\Delta f_i(x);
\end{equation}
\item the Nash embedding theorem guarantees an existence of an isometric embedding of $(M^N, g)$ into $\mathbb{R}^k$, where $k$ depends only on $N$. Moreover the canonical projection along the exponential map from a submanifold $\Phi(M^N)$ in $\mathbb{R}^k$ gives a typical example of $\pi$ as in (\ref{2}). In particular then $F_i=\Phi^{-1} \circ \pi \circ \Phi_i$ as in (\ref{b}) gives a diffeomorphism with (\ref{eq:bihl3}). 
\end{itemize}
As a corollary of Theorem \ref{thm:canonicaldiff} we have the following canonical sphere theorem which generalizes a result of \cite{WangZhao} to $\RCD$ spaces (see \cite{Perelmanjams} for the original work for sphere theorems of positive Ricci curvature in the smooth framework and see also \cite{Coldinga, Coldingb, Peterson} for outstanding contributions along this direction). Note that we do not assume $K=N-1$.
\begin{theorem}\label{thm:hemisphere}
For all $K \in \mathbb{R}$, $N \in \mathbb{N}$ and $\epsilon \in (0, 1)$, there exists $\delta=\delta(K, N, \epsilon) \in (0, 1)$ such that if an $\RCD(K, N)$ space $(X, \dist, \meas)$ satisfies
\begin{equation}
\dist_{\mathrm{GH}}\left( (X, \dist), (\mathbb{S}^N, \dist_{\mathbb{S}^N})\right)<\delta,
\end{equation}
where $\dist_{\mathrm{GH}}$ denotes the Gromov-Hausdorff distance, then a map $F:X \to \mathbb{S}^N$
\begin{equation}
F:=\left(\sum_{i=1}^{N+1}f_i^2\right)^{-1/2}\cdot \left(f_1,\ldots, f_{N+1}\right), \quad \text{where} \,\,\, \frac{1}{\meas_i(X_i)}\int_Xf_i^2\di \meas =\frac{1}{N+1}, \quad \forall i,
\end{equation}
gives an well-defined  homeomorphism and an $\epsilon$-Gromov-Hausdorff approximation with
\begin{equation}\label{eq:bihlip}
(1-\epsilon)\dist (x, y)^{1+\epsilon}\le \dist_{\mathbb{S}^n}(F(x), F(y)) \le (1+\epsilon) \dist(x, y),\quad \forall x, y \in X,
\end{equation}
where $f_i$ is an eigenfunction of $-\Delta$ with the $i$-th eigenvalue $\lambda_i$.
\end{theorem}
See \cite{Honda, HM, KM} for related sphere theorems in the $\RCD$ setting.
We also prove a similar result for flat tori. See Theorem \ref{thm:flat}.

The results above are proved by applying recent techniques in \cite{BrueNaberSemola, CheegerJiangNaber, WangZhao}, with slight modifications along \cite{AmbrosioHonda, AmbrosioHonda2, Honda}. In the next subsection let us provide an outline of the proof  of the main results.

\subsection{Strategy of proof and organization of the paper}
Let us focus on a simplified question here;
\begin{center}
For given (\ref{eq:mghconv2}), how to find a homeomorphism $F_i:X_i \to X$ satisfying
\end{center}
\begin{equation}\label{eq:biholderlip}
(1-\epsilon)\dist_i(x, y)^{1+\epsilon}\le \dist(F_i(x), F_i(y))\le (1+\epsilon)\dist_i (x, y),\quad \forall x, y \in X_i?
\end{equation}
Recall that $X$ is smooth and $N$-dimensional.
We first fix a smooth isometric embedding $\Phi=(\phi_1, \ldots, \phi_k):X \hookrightarrow \mathbb{R}^k$. Then thanks to stability results for Sobolev functions obtained in \cite{AmbrosioHonda, AmbrosioHonda2}, we can find a sequence of equi-regular maps $\Phi_i=(\phi_{i, 1},\ldots, \phi_{i,k}):X_i \to \mathbb{R}^k$ with equi-Lipschitz continuity of $\Delta_i\phi_{i, j}$ satisfying that $\Phi_i$ converge uniformly to $\Phi$. Then for any fixed $r \in (0, 1)$ we know
\begin{equation}\label{eq:smallav}
\sup_{x \in X_i}\frac{1}{\meas_i(B(x, r))}\int_{B(x, r)}\left|g_i-\Phi_i^*g_{\mathbb{R}^k}\right|\di \meas_i \to 0, \quad ( i \to \infty),
\end{equation}
where $B(x, r)$ denotes the open ball of radius $r$ centered at $x$. Combining (\ref{eq:smallav}) with a modified transformation result, Proposition \ref{thm:transformation},  (previously obtained in \cite{BrueNaberSemola, CheegerJiangNaber, WangZhao} for harmonic/eigen maps) for $\Phi_i$,  it is proved that  for small $\epsilon, r \in (0, 1)$,
\begin{equation}\label{eq;r}
(1-\epsilon)\dist_i(x, y)^{1+\epsilon}\le \left| \Phi_i(x)-\Phi_i(y)\right|_{\mathbb{R}^k}\le (1+\epsilon)\dist_i(x, y),\quad \forall x, y \in X_i \,\,\text{with $\dist_i(x, y) \le r$}
\end{equation}
for any sufficiently large $i$. Let us emphasize that the equi-Lipschitz continuity of $\Delta_i\phi_{i, j}$ plays a key role to get the sharp Lipschitz bound in (\ref{eq;r}). Thanks to the uniform convergence of $\Phi_i$ to $\Phi$, we can define maps $F_i:X_i \to X$ by $F_i:=\Phi^{-1}\circ \pi \circ \Phi_i$, where $\pi$ denotes the projection from a neighborhood of $\Phi(X)$ to $\Phi(X)$ along the exponential map from $\Phi(X)$ in $\mathbb{R}^k$. Since $\pi$ is almost $1$-Lipschitz locally, we have by (\ref{eq;r})
\begin{equation}\label{ddssx}
(1-\epsilon)\dist_i(x, y)^{1+\epsilon}\le \dist \left( F_i(x), F_i(y)\right)\le (1+\epsilon)\dist_i(x, y),\quad \forall x, y \in X_i \,\,\text{with $\dist_i(x, y) \le r$,}
\end{equation}
where the lower bound of (\ref{ddssx}) is justified by applying (\ref{eq;r}) for equi-regular maps $\Phi\circ F_i$. 
Since $F_i$ converge to $\mathrm{id}_X$, $F_i$ gives an $\epsilon_i$-Gromov-Hausdorff approximation for some $\epsilon_i \to 0^+$. Combining this with (\ref{ddssx}) proves (\ref{eq:biholderlip}). Finally it follows from invariance of domain that $F_i$ is a homeomorphism.

Note that all arguments above are justified by using a fact that $(X_i, \dist_i, \meas_i)$ and $(X, \dist, \meas)$ are non-collapsed in the sense of \cite{DePhillippisGigli}, up to multiplication by positive constants to the reference measures $\meas_i, \meas$, because we need a uniform Reifenberg flatness with respect to (\ref{eq:mghconv2}). This fact is guaranteed by \cite{H} (or recent \cite{BGHZ}).

The paper is organized as follows.
First of all, we emphasize that the proofs rely on heavily previously known analytical tools, namely we do not provide new substantial analytical results. However we believe that conclusions of the paper, in particular from the point of view of geometry, are interesting. Therefore in order to give a short presentation, assuming that readers are familiar with the basics of the theory of $\RCD$ spaces (a few facts for $\RCD$ spaces we need will be explained in section \ref{sec2}), we give the proof of Theorem \ref{thm:canonicaldiff} with technical results in sections \ref{section3}, \ref{section4} and \ref{section5}. As applications, we discuss the cases when the limit space is a sphere or a flat torus in section \ref{section6}, in particular Theorem \ref{thm:hemisphere} is proved. Finally in the appendix, section \ref{section7}, we provide a proof of a maximum principle, Proposition \ref{lem:subharm2}, as an independent interest, where this  for closed manifolds was a crucial role in \cite{WangZhao}.

\textbf{Acknowledgement.}
The both authors wish to thank Elia Bru\`e and Daniele Semola for fruitful discussions on the paper. They also thank Zhangkai Huang for giving us comments on the paper. Moreover they are grateful to the referee for his/her careful reading on the first version and for giving us very helpful comments.
The first named author acknowledges supports of the Grant-in-Aid for
Scientific Research (B) of 20H01799, the Grant-in-Aid for Scientific Research (B) of 21H00977 and Grant-in-Aid for Transformative Research Areas (A) of 22H05105.
\section{Preliminaries}\label{sec2}
Throughout the paper we will use standard notations in this topic. For example,
\begin{itemize}
\item denote by
\begin{equation}\label{eq:psi}
\Psi (\epsilon_1,\ldots, \epsilon_l|c_1,\ldots, c_m)
\end{equation}
a function $\Psi:(\mathbb{R}_{>0})^l\times \mathbb{R}^m \to (0, \infty)$ satisfying
\begin{equation}
\lim_{(\epsilon_1,\ldots, \epsilon_l) \to 0}\Psi (\epsilon_1,\ldots, \epsilon_l|c_1,\ldots, c_m)=0,\quad \forall c_i \in \mathbb{R}.
\end{equation}
Whenever we use the notation (\ref{eq:psi}), we immediately assume that $\epsilon_1, \ldots, \epsilon_l$ are sufficiently small, where the smallness are depending only on $c_1, \ldots, c_m$;
\item denote by $C(c_1, \ldots c_m)$ a positive constant depending only on constants $c_1,\ldots, c_m$, which may be changed from line to line in the sequel;
\item put
\begin{equation}
\intav_A\di \meas:=\frac{1}{\meas (A)}\int_A \di \meas;
\end{equation}
\item denote by $B(x, r)$ the open ball of radius $r \in (0, \infty)$ centered at $x$ and by $tB$ the ball $B(x, tr)$ for any $t \in (0, \infty)$ if $B=B(x, r)$.
\end{itemize}

In order to keep our presentation to be short, we assume that readers are familiar with basics on the theory of $\RCD$ spaces, for example, including;
\begin{itemize}
\item the definition of $\RCD(K, N)$ space $(X, \dist, \meas)$ for $K \in \mathbb{R}$ and $N \in [1, \infty]$;
\item the measured Gromov-Hausdorff convergence, the pointed version and their metrizations, $\dist_{\mathrm{mGH}}$, $\dist_{\mathrm{pmGH}}$;
\item the spaces of $L^p$-vector fields, $L^p$-tensor fields of type $(0, 2)$ over a Borel subset $A$ of an $\RCD(K, N)$ space, denoted by $L^p(T(A, \dist, \meas))$, $L^p((T^*)^{\otimes 2}(A, \dist, \meas))$, respectively;
\item $L^p$-weak/strong convergence of functions, vector fields, and tensor fields of type $(0, 2)$ with respect to the $\dist_{\mathrm{pmGH}}$-convergence.
\end{itemize}
We refer, for instance, \cite{A, AmbrosioGigliMondinoRajala, AmbrosioGigliSavare13, AmbrosioGigliSavare14, AmbrosioHonda, AmbrosioHonda2, AmbrosioMondinoSavare, CavallettiMilman, ErbarKuwadaSturm, Gigli13, Gigli1, Gigli, GigliMondinoSavare13, GP2} for the references. Let us here provide further details we will use later. In the sequel a real number $K \in \mathbb{R}$ and a finite $N \in [1, \infty)$ are fixed.

Let $(X, \dist, \meas)$ be an $\RCD(K, N)$ space. Assume that $X$ is not a single point. Then it is proved in \cite{BrueSemola} that there exists a unique $n \in \mathbb{N} \cap [1, N]$ such that the \textit{$n$-dimensional regular set}, denoted by $\mathcal{R}_n$, has the $\meas$-full measure, namely, for $\meas$-a.e. $x \in X$, we have
\begin{equation}\label{eq:regul}
\left(X, \frac{1}{r}\dist, \frac{1}{\meas(B(x, r))}\meas, x\right) \stackrel{\mathrm{pmGH}}{\to} \left(\mathbb{R}^n, \dist_{\mathbb{R}^n},\frac{1}{\omega_n}\mathcal{H}^n, 0_n\right),\quad (r \to 0^+),
\end{equation}
where $\dist_{\mathbb{R}^n}$ denotes the standard Euclidean distance and $\omega_n$ denotes the volume of a unit ball in $\mathbb{R}^n$ with respect to the $n$-dimensional Hausdorff measure $\mathcal{H}^n$. Call $n$ the \textit{essential dimension} of $(X, \dist, \meas)$ and denote by $\dim(X)$ it for short.

On the other hand there exists a unique $g \in L^{\infty}((T^*)^{\otimes 2}(X, \dist, \meas))$, called the \textit{canonical} Riemannian metric of $(X, \dist, \meas)$, such that for all $H^{1,2}$-Sobolev functions $f_i(i=1, 2)$ on $X$,
\begin{equation}
g(\nabla f_1, \nabla f_2)(x)=\langle \nabla f_1, \nabla f_2\rangle (x),\quad \text{for $\meas$-a.e. $x \in X$}, 
\end{equation}
where $\langle \cdot, \cdot \rangle $ denotes the $\meas$-a.e. pointwise inner product induced by the minimal relaxed slopes $|\nabla f_i|$.
Then the metric measure rectifiability of $(X, \dist, \meas)$ proved in \cite{DePhillippisMarcheseRindler, GigliPasqualetto, KellMondino} (after \cite{MondinoNaber}) allows us to prove via \cite[Theorem 5.1]{GP} (see also \cite[Lemma 3.3]{AHPT}) that 
\begin{equation}\label{eq:essdim}
|g|^2(x)=n,\quad \text{for $\meas$-a.e.  $x \in X$.}
\end{equation}
Moreover stability results for Sobolev functions with respect to the $\dist_{\mathrm{pmGH}}$-convergence established in \cite{AmbrosioHonda, AmbrosioHonda2, GigliMondinoSavare13} show that if a sequence of pointed $\RCD(K, N)$ spaces $(X_i, \dist_i, \meas_i, x_i)$ pointed measured Gromov-Hausdorff converge to $(X, \dist, \meas, x)$ for some $x \in X$, then $g_i$ $L^2_{\mathrm{loc}}$-weakly converge to $g$, where $g_i$ denotes the canonical Riemannian metric of $(X_i, \dist_i, \meas_i)$. Combining this with (\ref{eq:essdim}) easily yields the lower semicontinuity of the essential dimensions;
\begin{equation}\label{eq:loweress}
\liminf_{i \to \infty}\dim (X_i)\ge \dim (X)
\end{equation}
which is originally proved in \cite{Kita} by a different way. Note that after passing to a subsequence, the equality in (\ref{eq:loweress}) holds if and only if $g_i$ $L^2_{\mathrm{loc}}$-strongly converge to $g$.

Next let us recall two fundamental functional inequalities;
\begin{enumerate}
\item{(Poincar\'e inequality)} for any ball $B=B(x,r)$,
\begin{equation}
\intav_B\left|u-\intav_B u\d\m\right|\ \d\m\le 4re^{|K|r^2}\intav_{2B} |\nabla u|\ \d\m.
\label{eq1020}
\end{equation} 
Note that this is also valid for a more general class, called $\CD(K, \infty)$ spaces, see \cite{Rajala};
\item{(Bishop-Gromov inequality)} we have
\begin{equation}\label{eq:bg}
\frac{\meas(B(x, r))}{\mathrm{Vol}_{K, N}(r)} \downarrow \quad (r \uparrow 0),
\end{equation}
where $\mathrm{Vol}_{K, N}(r)$ denotes the volume of a ball of radius $r$ in the space form satisfying ``$\dim=N$ and $\mathrm{Ric}=N$''. Note that this is also valid for a more general class, called $\CD(K,N)$ spaces, see \cite{LottVillani, Sturm06, Sturm06b}.
\end{enumerate}
Since harmonic functions play key roles in this paper, let us recall the definition here. Denote by $D(\Delta)$ the domain of the Laplacian of $(X, \dist, \meas)$, namely $f \in D(\Delta)$ holds if and only if $f$ belongs to the $H^{1,2}$-Sobolev space $H^{1, 2}(X, \dist, \meas)$ and there exists a unique $h =\Delta f \in L^2(X, \meas)$ such that 
\begin{equation}
\int_X\langle \nabla f, \nabla \phi\rangle \di \meas=-\int_Xh\phi\di \meas,\quad \forall \phi \in H^{1,2}(X, \dist, \meas).
\end{equation}
For any open subset $U$ of $X$ denote by $D(\Delta, U)$ the set of all $f \in L^2(U, \meas)$ satisfying that $\phi f \in D(\Delta)$ with $|\nabla f|, \Delta f \in L^2(U, \meas)$ for any Lipschitz function $\phi$ on $U$ with compact support, where $|\nabla f|(x), \Delta f(x)$ make sense for $\meas$-a.e. $x \in U$ because of the locality of the minimal relaxed slope. Then a function $f$ on $U$ is said to be \textit{harmonic} if $f|_V \in D(\Delta, V)$ holds with $\Delta (f|_V)=0$ for any open subset $V \Subset U$. 

Finally let us recall a nice restricted class of $\RCD(K, N)$ spaces, called \textit{non-collapsed} spaces, introduced in \cite{DePhillippisGigli} (see also \cite{Kit1}), as a synthetic counterpart of (volume) non-collapsed Ricci limit spaces. 
We say that $(X, \dist, \meas)$ is \textit{non-collapsed} if $\meas=\mathcal{H}^N$. Then any non-collapsed spaces have finer properties rather than that of general $\RCD(K, N)$ spaces. For example if $(X, \dist, \meas)$ is a non-collapsed $\RCD(K, N)$ space, then
\begin{enumerate}
\item[3]{(Bishop inequality)} the quotient of  (\ref{eq:bg}) is bounded above by $1$;
\item[4]{(Maximality of essential dimension)} the essential dimension of $X$ is equal to $N$.
\end{enumerate}
It is worth pointing out that  the converse implication of the second statement above is also satisfied, namely
if an $\RCD(K, N)$ space $(X, \dist, \meas)$ has the essential dimension $N$, then $\meas=c\mathcal{H}^N$ for some $c \in (0, \infty)$. See \cite{BGHZ, H}.
Let us end this subsection by giving the following theorem which is a direct consequene of the observation above.  See for instance \cite{KM, HM} for the proof (see also \cite{AHPT}).
\begin{theorem}\label{thm:unifreifen}
Let 
\begin{equation}
(X_i, \dist_i, \meas_i) \stackrel{\mathrm{mGH}}{\to} (X, \dist, \meas)
\end{equation}
be a measured Gromov-Hausdorff convergent sequence of compact $\RCD(K, N)$ spaces. Assume that $(X, \dist, \meas)$ is non-collapsed. Then for any sufficiently large $i$ we have $\meas_i=c_i\mathcal{H}^N$ for some $c_i \in (0, \infty)$. Furthermore
if $X$ has no-singular points, namely, $X=\mathcal{R}_N$, then for any $\epsilon \in (0, \infty)$ there exist $r \in (0, 1]$ and $i_0 \in \mathbb{N}$ such that for all $i \ge i_0$, $x_i \in X_i$ and $t \in (0, r]$, we have
\begin{equation}\label{eq:reifflat10}
\dist_{\mathrm{pmGH}}\left( \left(X_i, \frac{1}{t}\dist_i, \frac{1}{\meas_i (B(x_i, t))}\meas_i, x_i\right), \left(\mathbb{R}^N, \dist_{\mathbb{R}^N}, \frac{1}{\omega_N}\mathcal{H}^N, 0_N\right)\right)<\epsilon.
\end{equation}
\end{theorem}
Note that under the same setting as in Theorem \ref{thm:unifreifen}, the intrinsic Reifenberg theorem proved in \cite{CheegerColding1} shows that for any $\alpha \in (0, 1)$, both $X_i$ and $X$ $\alpha$-bi-H\"older homeomorphic to an $N$-dimensional closed Riemannian manifold for any sufficiently large $i$.

\section{Transformation for regular splitting map}\label{section3}
Throughout the section, we fix;
\begin{itemize}
\item an $\RCD(K, N)$ space $(X, \dist, \meas)$ for some $K \in \mathbb{R}$ and some $N \in [1, \infty)$;
\item a positive constant $L \in [1, \infty)$;
\item a ball $B=B(x, 1)$ for some $x \in X$.
\end{itemize}
Note that although we consider only a fixed ball $B$ of radius $1$ as above in the sequel, the results we will prove below can be also rewritten for a general ball $B(x, r)$ of radius $r$ after a rescaling $r^{-1}\dist$, where we will use them later immediately.
Let us start this section by rewriting the definition of regular maps defined on $B$. See also \cite[Definition 3.2]{Honda}.
\begin{definition}[Regular map]\label{def:regular}
A map $F:B \to M^k$ from $B$ to a (not necessarily complete) $k$-dimensional Riemannian manifold $(M^k, g_{M^k})$ is said to be \textit{regular} if for all $\phi \in C^{\infty}(M^k)$, we have $\phi \circ F \in D(\Delta, B)$ with $\Delta (\phi \circ F) \in L^{\infty}(B, \meas)$.
\end{definition}
\begin{proposition}[Lipschitz continuity]\label{prop:regulargradient}
Any regular map is locally Lipschitz. 
More precisely this Lipschitz regularity comes from a fact that for a regular function $f:B \to \mathbb{R}$, if $X \setminus  B \neq \emptyset$, $|\Delta f| \le L$ on $B$, and 
\begin{equation}\label{eql1bd}
\intav_B\left|f-\intav_Bf\di \meas \right|\di \meas \le L,
\end{equation}
then $|\nabla f| \le C(K, N, L, t)$ on $tB$ for any $t \in (0, 1)$.
\end{proposition}
\begin{proof}
It is enough to check the assertion for $t=1/32$. Moreover thanks to \cite[Theorem 3.1]{Jiang} (see also \cite[Theorem 3.1]{AmbrosioMondinoSavare2}), it is also  enough to check that if $\int_Bf\di \meas =0$, then
\begin{equation}\label{eqinftty}
\|f\|_{L^{\infty}((1/4)B)}\le C(K, N, L).
\end{equation}
Find $\tilde{f} \in H^{1,2}_0((1/2)B, \dist, \meas) \cap D(\Delta, (1/2)B)$ with $\Delta \tilde{f}=\Delta f$ and 
\begin{equation}\label{eqlinfty}
\|\tilde{f}\|_{L^{\infty}((1/2)B)} \le C(K, N)\|\Delta f\|_{L^{\infty}((1/2)B)}
\end{equation}
(see \cite[Theorem 4.1]{BMosco}), where we used our assumption  $X \setminus  B \neq \emptyset$ (to get a Sobolev inequality, see for example \cite[(4.5)]{Cheeger}). Thus applying \cite[Lemma 4.1]{Jiang} for a harmonic function $f-\tilde{f}$ on $(1/2)B$ shows
\begin{equation}
\|f-\tilde{f}\|_{L^{\infty}((1/4)B)} \le C(K, N)\intav_{(1/2)B}|f-\tilde{f}|\di \meas.
\end{equation}
Combining this with (\ref{eqlinfty}) proves (\ref{eqinftty}).
\end{proof}
As a corollary of Proposition \ref{prop:regulargradient}, we have the following.
\begin{corollary}\label{lem:lipcont}
If a regular map $F=(f_1, \ldots, f_k):4B \to \mathbb{R}^k$ satisfies
\begin{equation}\label{eq:pullbackbd}
\intav_{4B}|\nabla f_i|^2\di \meas +\|\Delta f_i\|_{L^{\infty}}\le L,\quad \forall i,
\end{equation}
and $X \setminus 4B{\color{blue}\neq} \emptyset$,
then $F|_B$ is $C(K, N, L, k)$-Lipschitz, namely
\begin{equation}\label{liplip}
\left|F(y)-F(z)\right|_{\mathbb{R}^k}\le C(K, N, L, k)\dist (y, z),\quad \forall y, z \in B.
\end{equation}
\end{corollary}
\begin{proof}
Since (\ref{eq:pullbackbd}) with the Poincar\'e inequality implies 
\begin{equation}
\intav_{4B}\left| f-\intav_{4B}f\di \meas \right| \di \meas \le C(K, N)\intav_{4B}|\nabla f_i|^2\di \meas \le C(K, N)L,
\end{equation}
applying Proposition \ref{prop:regulargradient} shows 
\begin{equation}
|\nabla f_i|(y) \le C(K, N, L),\quad \forall y \in 3B.
\end{equation}
Then combining this with the local Sobolev-to-Lipschitz property obtained in
\cite[Proposition 1.10]{GV} proves (\ref{liplip})
\end{proof}

In the sequel, in addition, we also fix;
\begin{itemize}
\item a regular map $F=(f_1,\ldots, f_k):B \to \mathbb{R}^k$ for some $k \in \mathbb{N}$;
\item a positive number $\epsilon \in (0, 1)$.
\end{itemize} 
\begin{definition}[Transformation; definition]
If the matrix of size $k$:
\begin{equation}
G_{F, B}:=\left(\intav_B\langle \nabla f_i, \nabla f_j \rangle \di \meas \right)_{ij}
\end{equation}
is invertible, then the regular map $F_B=(f_{B, 1}, \ldots, f_{B, k}):B \to \mathbb{R}^k$ is defined as follows and called the \textit{canonical transformation} of $F$ on $B$; the Cholesky decomposition guarantees the unique existence of a lower triangular matrix $T_{F, B}$ of size $k$ whose diagonal consist of positive entries such that 
\begin{equation}
(T_{F, B})^t\cdot G_{F, B}\cdot T_{F, B}=E_k,
\end{equation}
where $E_k$ denotes the identity matrix of size $k$ and $A^t$ denotes the transpose of a matrix $A$. Then let 
\begin{equation}
F_B:=F \cdot (T_{F, B})^t.
\end{equation}
\end{definition}

It follows from the definition of $F_B$ that $G_{F_B, B}=E_k$ holds, namely,
\begin{equation}\label{eq:trans}
\left(\intav_B\langle \nabla f_{B, i}, \nabla f_{B, j}\rangle \di \meas\right)_{ij} =E_k.
\end{equation}
In the following lemma, compare with \cite[Lemma 3.15]{BrueNaberSemola}.
\begin{lemma}\label{lem:inductive}
We have the following.
\begin{enumerate}
\item\label{trans} 
If $G_{F, tB}$ and $G_{F_{tB}, sB}$ are invertible for some $s, t \in (0, 1]$, then $G_{F, sB}$ is also invertible with
\begin{equation}\label{eq:identity}
T_{F, sB}=T_{F, tB}\cdot T_{F_{tB}, sB}.
\end{equation}
\item\label{trivial} If 
\begin{equation}
\left|\intav_{B}\langle \nabla f_i, \nabla f_j\rangle -\delta_{ij}\di \meas \right|<\epsilon,\quad \forall i, j,
\end{equation}
then $G_{F, tB}$ is invertible for any $t \in [1/2, 1]$ with 
\begin{equation}
\left|T_{F, tB}-E_k\right| \le C(k)\epsilon.
\end{equation}
In particular if 
\begin{equation}
\intav_{B}\left|\langle \nabla f_i, \nabla f_j\rangle -\delta_{ij}\right|\di \meas <\epsilon,\quad \forall i, j,
\end{equation}
then  for any $t \in [1/2, 1]$
\begin{equation}
\intav_{B}\left| \langle \nabla f_{tB, i}, \nabla f_{tB, j}\rangle -\delta_{ij}\right| \di \meas\le  C(K, N, k)\epsilon,\quad \forall i, j.
\end{equation}
\item\label{induct} If   $G_{F, tB}$ is invertible for some $t \in [1/2, 1]$ with
\begin{equation}\label{small}
\intav_{tB}\left| \langle \nabla f_{tB, i}, \nabla f_{tB, j}\rangle -\delta_{ij}\right|\di \meas <\epsilon,\quad \forall i, j,
\end{equation}
then $G_{F, stB}$ is also invertible for any $s \in [1/2, 1]$ with
\begin{equation}
\intav_{t B}\left| \langle \nabla f_{stB, i}, \nabla f_{st B, j}\rangle -\delta_{ij}\right| \di \meas \le C(K, N, k)\epsilon
\end{equation}
and 
\begin{equation}\label{trans2}
\left| T_{F, tB} \cdot (T_{F, stB})^{-1}-E_k\right| + \left| T_{F, stB} \cdot (T_{F, tB})^{-1}-E_k\right| \le C(K, N, k)\epsilon.
\end{equation}
\end{enumerate}
\end{lemma}
\begin{proof}
It is elementary to check both (\ref{trans}) and (\ref{trivial}). Moreover (\ref{induct}) is a corollary of them. 
\end{proof}
\begin{corollary}\label{cor:trans}
For all $s, t \in (0, 1]$ with $s \le t$, if 
$G_{F, \alpha B}$ is invertible for any $\alpha \in [s, t]$ with
\begin{equation}\label{eq:closetoek}
\intav_{\alpha B}\left| \langle \nabla f_{\alpha B, i}, \nabla f_{\alpha B, j}\rangle -\delta_{ij}\right|\di \meas <\epsilon,\quad \forall i, j,
\end{equation}
then we have
\begin{equation}\label{eq:powerest}
\max \left\{ \left| (T_{F, tB})^{-1} \cdot T_{F, sB}\right|_{\infty}, \left| T_{F, sB} \cdot (T_{F, tB})^{-1}\right|_{\infty}  \right\} \le \left(\frac{t}{s}\right)^{C\epsilon},
\end{equation}
where $|\cdot|_{\infty}$ denotes the $L^{\infty}$-norm and $C=C(K, N, k)$.
\end{corollary}
\begin{proof}
It is enough to consider the case when $s =2^{-l-n}, t=2^{-l}$ for some $l, n \in \mathbb{N}$ because of (\ref{trans2}).
Recall 
\begin{equation}\label{eq:elemen}
\left|A_1A_2-E_k\right|_{\infty} \le \left| A_1-E_k\right|_{\infty} +\left|A_2-E_k\right|_{\infty}+k\left|A_1-E_k\right|_{\infty} \cdot \left|A_2-E_k\right|_{\infty}
\end{equation}
for all matrixes $A_i(i=1,2)$ of size $k$. 
Letting
\begin{equation}
a_m:=\left| (T_{F, 2^{-m}B})^{-1}\circ T_{F, 2^{-l-n}B}-E_k\right|_{\infty}, \quad  \forall m \in [l, l+n] \cap \mathbb{N},
\end{equation}
we have by (\ref{eq:elemen})
\begin{align}
a_m &\le \left| (T_{F, 2^{-m}B})^{-1} \cdot T_{F, 2^{-m-1}B}-E_k\right|_{\infty} +\left| (T_{F, 2^{-m-1}B})^{-1} \cdot T_{F, 2^{-l-n}B}-E_k\right|_{\infty} \nonumber \\
&\,+k\left| (T_{F, 2^{-m}B})^{-1} \cdot T_{F, 2^{-m-1}B}-E_k\right|_{\infty} \cdot \left| (T_{F, 2^{-m-1}B})^{-1}\cdot T_{F, 2^{-l-n}B}-E_k\right|_{\infty} \nonumber \\
&\le C\epsilon +a_{m+1} +k \cdot C \epsilon \cdot a_{m+1}\nonumber \\
&\le C\epsilon + (1+C\epsilon)a_{m+1}
\end{align}
which in particular implies
\begin{equation}\label{eq:polyno}
a_l+1 \le (1+C\epsilon)^n.
\end{equation}
Note that the left-hand-side of (\ref{eq:polyno}) is bounded below by $|(T_{F, 2^{-l}B})^{-1} \cdot T_{F, 2^{-l-n}B} |_{\infty}$.
Since $1+\log 2 \cdot t \le 2^{t}$ holds for any $t \in \mathbb{R}$, the right-hand-side of (\ref{eq:polyno}) is bounded above by $2^{Cn\epsilon}=(t/s)^{C\epsilon}$ which proves
\begin{equation}
\left| (T_{F, tB})^{-1} \cdot T_{F, sB}\right|_{\infty}\le \left(\frac{t}{s}\right)^{C\epsilon}.
\end{equation}
Similarly we get the desired estimate on $|T_{F, sB} \cdot (T_{F, tB})^{-1}|_{\infty}$. Thus we have 
 (\ref{eq:powerest}).
\end{proof}
We are now in a position to introduce the main result in this subsection. Compare with \cite[Proposition 3.13]{BrueNaberSemola} and \cite[Proposition 7.8]{CheegerJiangNaber}. 
\begin{proposition}[Transformation; existence]\label{thm:transformation}
If $(X, \dist, \meas)$ is an $\RCD(-\epsilon, N)$ space with
\begin{equation}
\dist_{\mathrm{pmGH}}\left(\left( X, \frac{1}{t}\dist, \frac{1}{\meas (tB)}\meas, x\right), \left(\mathbb{R}^n,  \dist_{\mathbb{R}^n}, \frac{1}{\omega_n}\mathcal{H}^n, 0_n\right)\right)<\epsilon, \quad \forall t \in (0, 1]
\end{equation}
for some $n \in \mathbb{N}$ and it holds that 
\begin{equation}
\intav_{B}\left|\langle \nabla f_i, \nabla f_j\rangle -\delta_{ij}\right| \di \meas<\epsilon,  \quad \|\Delta f_i\|_{L^{\infty}}\le L,\quad \quad \forall i, j,
\end{equation}
then  $G_{F, tB}$ is invertible for any $t \in (0, 1]$ with 
\begin{equation}\label{eq:transformat}
\intav_{3tB}\left| \langle\nabla f_{tB, i}, \nabla f_{tB, j}\rangle-\delta_{ij}\right|\di \meas <\Psi,
\end{equation}
where $\Psi =\Psi (\epsilon| N, L, k)$.
In particular on $B$
\begin{equation}\label{eq:laptrans}
t|\Delta f_{tB, i}| \le C(N, L, k)t^{1-\Psi}, \quad \forall i.
\end{equation}
\end{proposition}
\begin{proof}
The proof is done by a contradiction. Assume that the assertion is not satisfied. Then there exist a small
$\delta \in (0, 1)$ (the smallness depends only on $N$ and $k$)
and sequences of;
\begin{enumerate}
\item $\epsilon_i \to 0^+$;
\item pointed $\RCD(-\epsilon_i, N)$ spaces $(X_i, \dist_i, \meas_i, x_i)$ with 
\begin{equation}\label{eq:conveucld}
\dist_{\mathrm{pmGH}}\left(\left( X_i, \frac{1}{t}\dist_i, \frac{1}{\meas_i (tB_i)}\meas_i, x_i\right), \left(\mathbb{R}^n,  \dist_{\mathbb{R}^n}, \mathcal{H}^n, 0_n\right)\right)<\epsilon_i, \quad \forall t \in (0, 1],
\end{equation}
where $B_i=B(x_i, 1)$;
\item regular maps $F_i=(f_{i,1}, \ldots, f_{i, k}):B_i \to \mathbb{R}^k$ with
\begin{equation}
\intav_{B_i}\left|\langle \nabla f_{i, j}, \nabla f_{i, l}\rangle -\delta_{jl}\right| \di \meas<\epsilon_i, \quad \|\Delta_i f_{i, l}\|_{L^{\infty}}\le L,\quad \forall j, l
\end{equation}
\end{enumerate}
such that if $t_i$ denotes the infimum of $t \in (0, 1]$ satisfying
\begin{equation}
\intav_{3sB_i}\left| \langle\nabla f_{sB_i, j}, \nabla f_{sB_i, l}\rangle-\delta_{jl}\right|\di \meas <\delta, \quad \forall s \in [t, 1] \cap (0, 1/3],
\end{equation}
then
\begin{equation}
t_i>0
\end{equation}
holds, where the canonical transformation of $F_i$ on $sB_i$ is denoted by $F_{sB_i}=(f_{sB_i, 1},\ldots, f_{sB_i, k})$. Note that by definition of $t_i$, we see that
\begin{equation}
\intav_{3t_iB_i}\left| \langle\nabla f_{t_iB_i, j}, \nabla f_{t_iB_i, l}\rangle-\delta_{jl}\right|\di \meas \le \delta,\quad \forall j,l,
\end{equation}
(thus in particular $G_{F_i, t_iB_i}$ is invertible because of Lemma \ref{lem:inductive}) and that 
\begin{equation}\label{eq:notinvertible}
\intav_{3t_iB_i}\left| \langle\nabla f_{t_iB_i, j(i)}, \nabla f_{t_iB_i, l(i)}\rangle-\delta_{j(i)l(i)}\right|\di \meas =\delta
\end{equation}
for some $j(i), l(i)$, because of the continuity of the function;
\begin{equation}
s \mapsto \left(\intav_{3sB_i}\left| \langle\nabla f_{sB_i, j}, \nabla f_{sB_i, l}\rangle-\delta_{jl}\right|\di \meas \right)_{jl}.
\end{equation}

First let us prove:
\begin{equation}\label{eq:convzero}
t_i \to 0,\quad (i \to \infty).
\end{equation}
Since for any $t \in (0, 1]$ we have
\begin{align}\label{eq:monotone}
\intav_{tB_i}\left|\langle \nabla f_{i, j}, \nabla f_{i, l}\rangle -\delta_{jl}\right|\di \meas &\le \frac{\meas_i (B_i)}{\meas_i (tB_i)}\intav_{B_i}\left|\langle \nabla f_{i, j}, \nabla f_{i, l}\rangle -\delta_{jl}\right|\di \meas \nonumber \\
&\le C(N) \cdot \frac{1}{t^N}  \cdot \epsilon_i,
\end{align}
where we used the Bishop-Gromov inequality (\ref{eq:bg}), if $t \ge \epsilon_i^{1/(2N)}$, then (\ref{eq:monotone}) implies
\begin{equation}
\intav_{tB_i}\left|\langle \nabla f_{i, j}, \nabla f_{i, l}\rangle -\delta_{jl}\right|\di \meas \le C(N) \cdot \epsilon_i^{1/2} \to 0
\end{equation}
which proves (\ref{eq:convzero}) because of Lemma \ref{lem:inductive}.

Next let us consider rescaled spaces and maps;
\begin{equation}\label{eq:rescale}
(\tilde{X}_i, \tilde{\dist}_i, \tilde{\meas}_i, \tilde{x}_i):=\left(X_i, \frac{1}{2t_i}\dist_i, \frac{1}{\meas_i(2t_iB_i)}\meas_i, x_i\right),
\end{equation}
\begin{equation}
\tilde{F}_i:=\frac{1}{2t_i} \cdot \left(F_{2t_iB_i}-F_{2t_iB_i}(x_i) \right),
\end{equation}
where we will use similar notations for rescaled objects, e.g. $\tilde{\Delta}_i$ for the Laplacian and $\tilde{B}_i$ for the unit ball centered at $\tilde{x}_i$ etc. on the rescaled space (\ref{eq:rescale}).

Denoting by $\tilde{F}_i=(\tilde{f}_{i, 1},\ldots, \tilde{f}_{i, k})$ and $\tilde{B}_i=\tilde{B} (\tilde{x}_i, 1) (=2t_iB_i)$, $\tilde{\Delta}_i\tilde{f}_{i, j}$ can be estimated uniformly on $(2t_i)^{-1}\tilde{B}_i (=B_i)$ as follows;
\begin{align}\label{eq:almostharm}
\left| \tilde{\Delta}_i\tilde{f}_{i, j}\right| = 2t_i\left|\Delta_if_{2t_iB_i, j} \right| &\le k \cdot 2t_i \cdot |T_{F_i, 2t_iB_i}|_{\infty} \cdot \max_l |\Delta_i f_{i, l}| \nonumber \\
&\le k \cdot 2t_i \cdot \left(t_i\right)^{-C(N, k)\delta}  \cdot L \to 0,\quad (i \to \infty),
\end{align}
where we used Corollary \ref{cor:trans} and a fact that $C(N, k)\delta<1$.

Next let us prove that 
\begin{itemize}
\item for any $R \in [1, \infty)$, there exists $i_0 \in \mathbb{N}$ such that for all $i \ge i_0$ and  $\tilde{y}_i \in (2t_i)^{-1}\tilde{B}_i (=B_i)$ with $1 \le \tilde{\dist}_i(\tilde{x}_i, \tilde{y}_i)\le R$, we have 
\begin{equation}\label{eq:lineargrwoth}
\left|\tilde{f}_{i, j}(\tilde{y}_i)\right| \le C(N)\tilde{\dist}_i(\tilde{x}_i, \tilde{y}_i)^{1+C(N, k)\delta}.
\end{equation}
\end{itemize}
The proof is as follows.
Letting $\alpha_i:=\dist_i(x_i, \tilde{y}_i) =2t_i\tilde{\dist}_i(\tilde{x}_i, \tilde{y}_i) (\in [2t_i, 2Rt_i])$, we have
\begin{align}\label{al:growth}
|\tilde{f}_{i, j}(\tilde{y}_i)| &= \frac{1}{2t_i}\left|F_{2t_iB_i}(\tilde{y}_i)-F_{2t_iB_i}(x_i)\right| \nonumber \\
&\le \frac{1}{2t_i}\left| T_{F_i, 2t_iB_i}\cdot T_{F_i, 4\alpha_iB_i}^{-1}\right|_{\infty} \cdot \left| F_{4\alpha_iB_i}(\tilde{y}_i)-F_{4\alpha_iB_i}(x_i)\right| \nonumber\\
&\le \frac{1}{2t_i}\cdot \left(\frac{2\alpha_i}{t_i}\right)^{C(N, k)\delta} \cdot \left| F_{4\alpha_iB_i}(\tilde{y}_i)-F_{4\alpha_iB_i}(x_i)\right|.
\end{align}
On the other hand, by definition of $t_i$ with 
Corollary \ref{cor:trans}, we have
\begin{align}\label{ee}
4\alpha_i |\Delta_i f_{4\alpha_iB_i, j}(\tilde{y}_i)| &\le  4\alpha_i \cdot |T_{F_i,  4\alpha_iB_i}|_{\infty} \cdot k \cdot \max_l |\Delta_i f_{i, l}(\tilde{y}_i)| \nonumber \\
&\le 4kL\alpha_i^{1-C(N, k)\delta}\le 4kL(2Rt_i)^{1-C(N,k)\delta}.
\end{align}
Consider the rescaled objects again;
\begin{equation}\label{100}
\left( \hat{X}_i, \hat{\dist}_i, \hat{\meas}_i\right):=\left( X_i, \frac{1}{4\alpha_i}\dist_i, \frac{1}{\meas_i(4\alpha_iB_i)}\meas_i\right), \quad \hat{F}_i:= \frac{1}{4\alpha_i} \cdot F_{4\alpha_iB_i}
\end{equation}
Then applying Corollary \ref{lem:lipcont}  for (\ref{100}) shows that there exists $i_0 \in \mathbb{N}$ such that for any $i \ge i_0$ we have
\begin{equation}
\left| F_{4\alpha_iB_i}(\tilde{y}_i)-F_{4\alpha_iB_i}(x_i)\right|\le C(N) \cdot \dist_i(x_i, \tilde{y}_i).
\end{equation}
Thus the right hand side of (\ref{al:growth}) is bounded above by
\begin{equation}
\frac{1}{2t_i} \cdot \tilde{\dist}_i(x_i, y_i)^{C(N, k)\delta} \cdot C(N) \cdot \dist_i(x_i, y_i) = C(N) \cdot \tilde{\dist}_i(x_i, y_i)^{1+C(N,k)\delta}
\end{equation}
which proves (\ref{eq:lineargrwoth}).


Recalling by  (\ref{eq:conveucld}) 
\begin{equation}
(\tilde{X}_i, \tilde{\dist}_i, \tilde{\meas}_i, \tilde{x}_i) \stackrel{\mathrm{pmGH}}{\to} \left(\mathbb{R}^n, \dist_{\mathbb{R}^n}, \frac{1}{\omega_n}\mathcal{H}^n, 0_n\right),
\end{equation}
applying stability results in \cite[Theorem 4.4]{AmbrosioHonda2} with  (\ref{eq:almostharm}) and (\ref{eq:lineargrwoth}), after passing to a subsequence, there exist a harmonic map $\tilde{F}=(\tilde{f}_1, \ldots, \tilde{f}_k):\mathbb{R}^n \to \mathbb{R}^k$ such that $\tilde{f}_{i, j}$ $H^{1,2}_{\mathrm{loc}}$-strongly converge to $\tilde{f}_j$ for any $j$ and that for any $w \in \mathbb{R}^n$ with $|w|_{\mathbb{R}^n}\ge 1$,
\begin{equation}
|\tilde{f}_j(w)| \le C(N)|w|_{\mathbb{R}^n}^{1+C(N, k)\delta}.
\end{equation}
Note that any harmonic function on $\mathbb{R}^n$ with  at most $(1+ \alpha)$-polynomial growth for some $\alpha \in [0, 1)$ is linear.
Thus $\tilde{F}$ is linear because of  $C(N, k)\delta<1$.

On the other hand, denoting by $B_{\mathbb{R}^n}:=B(0_n, 1)$, (\ref{eq:trans}) yields 
\begin{align}
\left({\color{blue}\intav_{B_{\mathbb{R}^n}}}\langle \nabla \tilde{f}_j, \nabla \tilde{f}_l\rangle \di \mathcal{H}^N\right)_{jl}=\lim_{i \to \infty}\left(\int_{\tilde{B}_i}\langle \nabla \tilde{f}_{i, j}, \nabla \tilde{f}_{i, l}\rangle \di \tilde{\meas}_i\right)_{jl} =E_k.
\end{align}
In particular 
\begin{equation}
\left( \langle \nabla \tilde{f}_j, \nabla \tilde{f}_l\rangle\right)_{jl}\equiv E_k.
\end{equation}
Thus since 
\begin{align}
\left(\intav_{{\color{blue}3}t_iB_i} \left| \langle \nabla f_{2t_iB_i, j}, \nabla f_{2t_iB_i, l}\rangle -\delta_{jl} \right| \di \meas_i\right)_{jl} \to \left(\intav_{({\color{blue}3}/2)B_{\mathbb{R}^n}}\left|\langle \nabla \tilde{f}_j, \nabla \tilde{f}_l\rangle-\delta_{jl}\right|\di \mathcal{H}^N\right)_{jl}=0,
\end{align}
applying Lemma \ref{lem:inductive} shows that 
\begin{equation}
\left(\intav_{{\color{blue}3}t_iB_i}\left|\langle \nabla f_{t_iB_i, j}, \nabla f_{t_iB_i, l}\rangle -\delta_{jl}\right|\di \meas_i\right)_{jl} \to 0
\end{equation}
which contradicts  (\ref{eq:notinvertible}).

The remaining statement, (\ref{eq:laptrans}), is a corollary of (\ref{eq:transformat}) with Corollary \ref{cor:trans}.
\end{proof}
\begin{remark}
As discussed in \cite[Section 7]{CheegerJiangNaber}, the arguments above still work if we replace $\mathbb{R}^n$ by a metric measure cone (e.g. the half space $\mathbb{R}^n_+$, see \cite[Lemma 3.14]{BrueNaberSemola}), or more generally by a family of metric measure cones having no non-linear harmonic functions with at most $(1+\alpha)$-polynomial growth, for a fixed $\alpha \in (0, \infty)$.
\end{remark}
\begin{remark}
Let us consider Proposition \ref{thm:transformation} in a typical case, namely $(X, \dist, \meas)$ is isometric to $(\mathbb{R}^n, \dist_{\mathbb{R}^n}, \mathcal{H}^n)$. Then it is known that if $\|\Delta f\|_{L^{\infty}}<\infty$ holds, then $f$ is $C^1$. Thus in this typical case the conclusion of the theorem above is trivial. Note that this $C^1$-regularity is, for example, a direct consequence of the \textit{$L^p$-Calder\'on-Zygmund inequality} (denoted by $\mathrm{CZ}(p)$ for short) for large $p \in (1, \infty)$;
\begin{equation}\label{eq:cz}
\|\mathrm{Hess}_f\|_{L^p} \le C \left(\|\Delta f\|_{L^p} + \|\nabla f\|_{L^p}\right)
\end{equation}
for any $f \in C^{\infty}_c(\mathbb{R}^n)$ (note that the $\mathrm{CZ}(p)$ is valid for any closed Riemannian manifold and any $p \in (1, \infty)$ for some constant $C \in (0, \infty)$ which depends on $p$ and $g$). However in general for $\RCD(K, N)$ spaces, the continuity of $|\nabla f|$ is not satisfied even for harmonic functions. Namely, such $C^1$-regularities do not hold on general $\RCD(K, N)$ spaces. Along this direction, it is also known that in general the $\mathrm{CZ}(p)$ is not satisfied for $p>2$ in the $\RCD$ setting. See \cite[Corollary 1.3]{DZ} (and see also \cite{HMRV, L, MV, P}).
\end{remark}

\section{Locally Lipschitz-H\"older embedding via regular map}\label{section4}
Throughout this section we will fix;
\begin{itemize}
\item a pointed $\RCD(K, N)$ space $(X, \dist, \meas, x)$ with the essential dimension $n$ for some $K \in \mathbb{R}$ and some $N \in [1, \infty)$. Denote by $g$ the canonical Riemannian metric of $(X, \dist, \meas)$;
\item positive numbers $k, l \in \mathbb{N}$ and $L \in [1, \infty)$;
\item positive small numbers $\epsilon, r \in (0, 1)$; 
\item an open ball $B=B(x, 1)$.
\end{itemize}
\subsection{Pull-back and regular splitting map}
\begin{definition}[Pull-back]\label{definition}
For any  Lipschitz map $F:A \to M^k$ from a Borel subset $A$ of $X$ to a $k$-dimensional (not necessarily complete) Riemannian manifold $(M^k, g_{M^k})$, the \textit{pull-back} $F^*g_{M^k}$ by $F$ is defined by
\begin{equation}\label{eq:defpull}
F^*g_{M^k}=\sum_{i=1}^l\dist \tilde{f}_i\otimes \dist \tilde{f}_i \in L^{\infty}((T^*)^{\otimes 2}(A, \dist, \meas)),
\end{equation}
where $\Phi:M^k \hookrightarrow \mathbb{R}^l$ is a smooth isometric (namely, $\Phi^*g_{\mathbb{R}^l}=g_{M^k}$) embedding and $\Phi \circ F=(\tilde{f}_1,\ldots, \tilde{f}_l)$.
\end{definition}
In the definition above, it is easy to see that the pull-back does not depend on the choice of isometric embeddings $\Phi$, thus (\ref{eq:defpull}) is well-defind (see also \cite{Honda, HS}).
The main purpose of this section is to discuss the case when $M^k=\mathbb{R}^k$. In particular we study a regular map $F:B \to \mathbb{R}^k$ satisfying;
\begin{equation}\label{eqalmosti}
\intav_{B}\left|g-F^*g_{\mathbb{R}^k}\right|\di \meas<\epsilon.
\end{equation}
In connection with this motivation, let us give a simple observation for \textit{splitting maps} whose 

original definition in the smooth framework can be found in \cite[Theorem 6.62]{CheegerColding} and \cite[Lemma 1.23]{Colding} for harmonic maps into Euclidean spaces (see \cite[Definition 0.1]{BPS} in the RCD framework).

For our purpose, we will deal with a slightly generalized formulation as discussed in the sequel (in particular maps are not necessary to be harmonic). Roughly speaking the following proposition tells us that splitting maps provide typical examples of (\ref{eqalmosti}).
\begin{proposition}\label{prop:splittingexample}
If $(X, \dist, \meas)$ is an $\RCD(-\epsilon, N)$ space and a map $F=(f_1, \ldots, f_n):6nB \to \mathbb{R}^n$ satisfies $\Delta f_i \in D(\Delta, 6nB)$ (recall that the essential dimension is denoted by $n$),
\begin{equation}
\intav_{6nB}\left((\Delta f_i)^2+\left| \langle \nabla f_i, \nabla f_j\rangle -\delta_{ij}\right|\right)\di \meas<\epsilon, \quad \|\nabla f_i\|_{L^{\infty}} \le L, \quad \forall i, j,
\end{equation}
then it holds that
\begin{equation}\label{eq:l1close}
\intav_{B}\left|g-F^*g_{\mathbb{R}^n}\right|\di \meas<\Psi
\end{equation} 
and that $F|_{B}$ gives a $\Psi$-Gromov-Hausdorff approximation to the image which is also $\Psi$-Hausdorff close to $B(F(x), 1)$,
where  $\Psi=\Psi(\epsilon| N, L)$.
\end{proposition}
\begin{proof}
Let us prove only (\ref{eq:l1close}) because the proofs of remaining statements are similar.
The proof is done by a contradiction. Assume that the assertion is not satisfied. Then there exist $\tau \in (0, 1)$ and sequences of 
\begin{itemize}
\item positive numbers $\epsilon_i \to 0^+$;
\item pointed $\RCD(-\epsilon_i, N)$ spaces $(X_i, \dist_i, \meas_i, x_i)$ with the essential dimension $n$ and $\meas_i(6nB_i)=1$, where $B_i=B(x_i, 1)$;
\item maps $F_i=(f_{i, 1}, \ldots, f_{i, n}):6nB_i \to \mathbb{R}^n$ with $f_{i, j} \in D(\Delta_i, 6nB_i)$
\end{itemize}
such that 
\begin{equation}
\|\nabla_if_{i, j}\|_{L^{\infty}}\le L,\quad \intav_{6nB_i}\left( (\Delta_i f_{i, j})^2 +\left| \langle \nabla f_{i, j}, \nabla f_{i, k}\rangle -\delta_{jk}\right|\right)\di \meas_i \le \epsilon_i,\quad \forall i, j, k
\end{equation}
and
\begin{equation}\label{eq:lowerbd}
\intav_{B_i}\left|g_i-F_i^*g_{\mathbb{R}^n}\right|\di \meas_i \ge \tau
\end{equation}
are satisfied, where $g_i$ denotes the canonical Riemannian metric of $(X_i, \dist_i, \meas_i)$. By \cite[Theorem 4.4]{AmbrosioHonda2} after passing to a subsequence with no loss of generality we can assume that $(X_i, \dist_i, \meas_i, x_i)$ pointed measured Gromov-Hausdorff converge to a pointed $\RCD(0, N)$ space $(X, \dist, \meas, x)$ and that $F_i$ $H^{1,2}$-strongly converge to a harmonic map $F=(f_1, \ldots, f_n):6nB \to \mathbb{R}^n$ on each compact subset of $6nB$, with $\langle \nabla f_i, \nabla f_j\rangle \equiv \delta_{ij}$, where $B=B(x, 1)$. Thus applying the local splitting theorem obtained in \cite[Theorem 3.4]{BrueNaberSemola}, there exist a metric measure space $(Z, \dist_Z, \meas_Z)$ and a map $f:nB \to Z$ such that the map
\begin{equation}\label{eq:locspltmap}
(f_1, \ldots, f_n, f): B \to \mathbb{R}^n \times Z
\end{equation}
gives an isometry between metric measure spaces with its image including $B(F(x), 1) \times \{f(x)\}$.

Let us check that $f(B)$ is a single point as follows. If $f(B)$ is not a single point, then there exists $y \in B$ such that $f(x) \neq f(y)$ holds. Taking a minimal geodesic $\gamma$ from the center $x$ of $B$ to $y$, and then taking a tangent cone at $x$, it follows from the isometry in (\ref{eq:locspltmap}) that the tangent cone includes the $\mathbb{R}^n$-factor and that $\gamma$ converges to a ray in the tangent cone whose image is not included in the $\mathbb{R}^n$-factor. In particular the essential dimension of the tangent cone is at least $n+1$, which contradicts with  the lower semicontinuity of essential dimensions proved in \cite[Theorem 1.5]{Kita}. Thus $f(B)$ is a single point.


In particular the map
\begin{equation}
F=(f_1, \ldots, f_n): B\to \mathbb{R}^n
\end{equation}
gives an isometry with its image $B(F(x), 1)$. In particular $g_i$ $L^{2}_{\mathrm{loc}}$-strongly converge to the canonical Riemannian metric $g$ of $(X, \dist, \meas)$. Thus we have
\begin{equation}
\intav_{B_i}\left| g_i-F_i^*g_{\mathbb{R}^n}\right|\di \meas_i \to \intav_{B}\left| g-F^*g_{\mathbb{R}^n}\right| \di \meas =0
\end{equation}
which contradicts (\ref{eq:lowerbd}). 
\end{proof}
\subsection{Sharp Lipschitz upper bound}
The main purpose of this subsection is to establish a sharp Lipschitz upper bound for a regular map $F$;
\begin{equation}
\left| F(y)-F(z)\right|_{\mathbb{R}^k} \le (1+\epsilon)\dist(y, z), \quad \forall y, z \in B
\end{equation}
under assuming that the average of $|g-F^*g_{\mathbb{R}^k}|$ is small. The original argument as in the proof of the following lemma can be found in \cite[page $21$]{CheegerNaber} for harmonic functions in the smooth framework.
\begin{lemma}\label{lem:sharpharm}
If $(X, \dist, \meas)$ is an $\RCD(-\epsilon, N)$ space and a regular function  $f :4B\to \mathbb{R}$ satisfies
\begin{equation}
\|\nabla f\|_{L^{\infty}}\le L,\quad \|\nabla \Delta f\|_{L^{\infty}}+\intav_{4B}\left| |\nabla f|^2-\lambda^2 \right|\di \meas <\epsilon
\end{equation}
for some $\lambda \in [0, L]$, then
\begin{equation}\label{eewwrrrt}
|\nabla f|(z) \le \lambda +C(N, L)\epsilon,\quad \text{for $\meas$-a.e. $z \in 3B$}.
\end{equation} 
In particular $f|_B$ is $(\lambda +C(N, L)\epsilon)$-Lipschitz, namely
\begin{equation}\label{eq:localsl}
|f(z)-f(y)|\le (\lambda +C(N, L)\epsilon)\dist (y, z), \quad \forall y, z \in B.
\end{equation}
\end{lemma}
\begin{proof}
Let us follow the same arguments in \cite[Remark 3.3]{BrueNaberSemola}. Without loss of generality we can assume $\meas(4B)=1$.
Take a good cut-off $\phi \in D(\Delta)$ with $0 \le \phi \le 1$, $\phi=1$ on $(7/2)B$, $\phi =0$ outside $4B$ and $|\nabla \phi| +|\Delta \phi| \le C(N)$ (see \cite[Lemma 3.1]{MondinoNaber} for the existence).
For any $z \in 3B$, recalling the Bochner inequality (without the Hessian term);
\begin{align}\label{eq:bochsharp}
&\int_{4B}\Delta_y(\phi (y) p(z, y, t))  \cdot \left(|\nabla f|^2(y)-\lambda^2\right)\di \meas(y) \nonumber \\
&\ge 2\int_{4B}\phi(y)p(z, y, t)\langle \nabla_y\Delta_yf(y), \nabla_yf(y)\rangle \di \meas(y) -2\epsilon\int_{4B}\phi(y)p(z, y, t)|\nabla f|^2(y)\di \meas(y), 
\end{align}
where $p$ denotes the heat kernel of $(X, \dist, \meas)$.
First let us prove that the right-hand-side of (\ref{eq:bochsharp}) is almost non-negative as follows.

Since we have
\begin{equation}
\left|\int_{4B}\phi(y)p(z, y, t)\langle \nabla_y\Delta_yf(y), \nabla_yf(y)\rangle \di \meas(y)\right| \le \epsilon L \int_X p(z, y, t)\di \meas(y) = \epsilon L
\end{equation}
and 
\begin{equation}
\left| \int_{4B}\phi(y)p(z, y, t)|\nabla f|^2(y)\di \meas(y)\right| \le L^2\int_Xp(z, y, t)\di \meas(y) =L^2,
\end{equation}
the right-hand-side of (\ref{eq:bochsharp}) is bounded below by $-4\epsilon L^2$.

On the other hand, the left-hand-side of (\ref{eq:bochsharp}) is equal to
\begin{align}\label{eq:lhs2}
&\int_{4B}\left(|\nabla f|^2(y)-\lambda^2\right)\phi(y)\Delta_yp(z, y, t)\di \meas(y) + \int_{4B}\left(|\nabla f|^2(y)-\lambda^2\right)\Delta_y\phi(y)  \cdot p(z, y, t)\di \meas(y) \nonumber \\
&+2\int_{4B}\left(|\nabla f|^2(y)-\lambda^2\right)\langle \nabla_y\phi(y),  \nabla_y p(z, y, t)\rangle \di \meas(y).
\end{align}
Then let us prove that the second and the third terms of (\ref{eq:lhs2}) are small as follows. Applying Gaussian estimates obtained in \cite[Theorem 1.2, Corollary 1.2]{JiangLiZhang} (see also \cite{GarofaloMondino, Jiang15}) we have
\begin{align}\label{eq:secondest}
&\left| \int_{4B}\left(|\nabla f|^2(y)-\lambda^2\right)\Delta_y\phi(y)  \cdot p(z, y, t)\di \meas(y)\right| \nonumber \\
&=\left| \int_{4B\setminus (7/2)B}\left(|\nabla f|^2(y)-\lambda^2\right)\Delta_y\phi(y)  \cdot p(z, y, t)\di \meas(y)\right| \nonumber \\
&\le C(N)\int_{4B \setminus (7/2)B}\left| |\nabla f|^2(y)-\lambda^2\right| \cdot \frac{1}{\meas (B(z, t^{1/2}))} \cdot \exp \left(-\frac{\dist(z, y)^2}{5t}\right)\di \meas (y) \nonumber \\
&\le C(N)\int_{4B \setminus (7/2)B}\left| |\nabla f|^2(y)-\lambda^2\right| \cdot \frac{1}{t^{N/2}} \cdot \exp \left(-\frac{1}{50t}\right)\di \meas (y) \nonumber \\
&\le C(N)\int_{4B}\left| |\nabla f|^2(y)-\lambda^2\right| \di \meas(y) \nonumber \\
&\le C(N) \epsilon,
\end{align}
where we used the fact that the function, $t \mapsto t^{-N/2}\exp (-50t^{-1})$, is bounded in $(0, 1]$.
Similarly 
\begin{align}
&\left| \int_{4B}\left(|\nabla f|^2(y)-\lambda^2\right)\langle \nabla_y\phi(y),  \nabla_y p(z, y, t)\rangle \di \meas(y)\right| \nonumber \\
&=\left| \int_{4B \setminus (7/2)B}\left(|\nabla f|^2(y)-\lambda^2\right)\langle \nabla_y\phi(y),  \nabla_y p(z, y, t)\rangle \di \meas(y)\right| \nonumber \\
&\le C(N) \int_{4B\setminus (7/2)B} \left||\nabla f|^2(y)-\lambda^2\right| \cdot \frac{1}{t^{1/2}\meas(B(z, t^{1/2}))}\cdot \exp \left(-\frac{\dist(z, y)^2}{5t}\right)\di \meas(y) \nonumber \\
&\le C(N)  \epsilon.
\end{align}
In particular, letting 
\begin{equation}
f_t(z):=\int_X\left(|\nabla f|^2(y)-\lambda^2\right)\phi(y)p(z, y, t)\di \meas(y),
\end{equation}
we have
\begin{equation}
\frac{\di}{\di t}f_t(z)\ge -C(N, L)\epsilon.
\end{equation}
Thus integrating this over $[0, 1]$ with respect to $t$ shows
\begin{equation}
f_1(z)-\left(|\nabla f|^2(z)-\lambda^2\right) \ge -C(N, L)\epsilon,\quad \text{for $\meas$-a.e. $z \in 3B$.}
\end{equation}
Finally we have (\ref{eewwrrrt}) since $|f_1(z)|$ is small because of a similar argument as in (\ref{eq:secondest}). The remaining Lipschitz continuity is justified from this with the local Sobolev-to-Lipschitz property \cite[Proposition 1.10]{GV}. 
\end{proof}

\begin{lemma}\label{lem:sharplipbd}
If $(X, \dist, \meas)$ is an $\RCD(-\epsilon, N)$ space and a regular map $F=(f_1,\ldots, f_k):4B \to \mathbb{R}^k$ satisfies
\begin{equation}
\intav_{4B}\left|g-F^*g_{\mathbb{R}^k}\right|\di \meas +\|\nabla \Delta f_i\|_{L^{\infty}}<\epsilon, \quad \|\nabla f_i\|_{L^{\infty}} \le L, \quad \forall i,
\end{equation}
then $F|_{B}$ is $(1+\Psi(\epsilon | N, L, k))$-Lipschitz.
\end{lemma}
\begin{proof}
Let 
\begin{equation}
A:=\left\{ y \in 4B \Big| \left|  g-F^*g_{\mathbb{R}^k}\right|(y) \le \epsilon^{1/2} \right\}.
\end{equation}
Then we have
\begin{equation}\label{eq:measest}
\meas (4B \setminus A) \le \frac{1}{\epsilon^{1/2}}\int_{4B\setminus A} \left|g-F^*g_{\mathbb{R}^k}\right|\di \meas \le \frac{1}{\epsilon^{1/2}}\int_{4B} \left|g-F^*g_{\mathbb{R}^k}\right|\di \meas \le \epsilon^{1/2} \cdot \meas (4B)
\end{equation}
and
\begin{equation}\label{eq:pointwiseest}
\left| \langle \nabla f_i, \nabla f_j\rangle(y)-\sum_{l=1}^k\langle \nabla f_i, \nabla f_l\rangle(y) \cdot \langle \nabla f_j, \nabla f_l\rangle (y) \right|\le \Psi (\epsilon| L, k), \quad \text{for $\meas$-a.e. $y \in A$}
\end{equation}
for all $i, j$.
In particular (\ref{eq:measest}) and (\ref{eq:pointwiseest}) easily imply
\begin{equation}
\left| G_{F, 4B}-(G_{F, 4B})^2\right|<\Psi
\end{equation}
where $\Psi=\Psi(\epsilon|N, L, k)$.
Thus any eigenvalue of $G_{F, 4B}$ is in a $\Psi$-neighborhood of  a discrete set $\{0, 1\}$. Let us denote by $l \in \mathbb{N}$ the number of eigenvalues $\lambda$ of $G_{F, 4B}$ satisfying that $\lambda$ is $\Psi$-close to $1$. Note that without loss of generality, we may assume $\Psi<1/2$, thus the number $l$ is well-defined. 
Then since 
\begin{equation}
\left| l-\intav_{4B}\left\langle F^*g_{\mathbb{R}^k}, g \right\rangle\di \meas\right| \le \Psi,
\end{equation}
we have by (\ref{eq:essdim})
\begin{align}
\left| l-n\right| &\le \left| \intav_{4B} \left\langle F^*g_{\mathbb{R}^k} -g, g \right\rangle\di \meas\right| +\Psi \le \Psi<\frac{1}{2}
\end{align}
which implies $l=n$.

Denote by $\{\lambda_i\}_{i=1}^k$ the eigenvalues of $G_{F, 4B}$ with
\begin{equation}
1 -\Psi\le \lambda_1 \le \lambda_2 \le \cdots \le \lambda_n \le 1+\Psi,  \quad -\Psi \le \lambda_{n+1} \le \cdots \le \lambda_k \le \Psi.
\end{equation}
Moreover let us find $P \in O(k)$ satisfying that $P^tG_{F, 4B}P$ is a diagonal matrix whose $i$-th entry is equal to to $\lambda_i$ for any $i$.
Define $F_P:=FP^t$ and $F_P^n:=(f_{P, 1}, \ldots, f_{P, n})$, where $F_P=(f_{P, 1},\ldots, f_{P, k})$. 
Then it is easy to see
\begin{equation}\label{eq:splitsplit}
\intav_{4B}\left| \langle \nabla f_{P, i}, \nabla f_{P, j}\rangle -\delta_{ij} \right|\di \meas<\Psi, \quad \forall i, j \le n
\end{equation}
and
\begin{equation}\label{eq:small}
\intav_{4B}| \nabla f_{P, i}|^2\di \meas <\Psi,\quad \forall i \ge n+1.
\end{equation}
In particular Lemma \ref{lem:sharpharm} with (\ref{eq:small}) shows that $f_{P, i}|_B$ is $\Psi$-Lipschitz for any $i \ge n+1$.
Thus
\begin{equation}\label{eq:1pmlip}
\left| F(y)-F(z)\right|_{\mathbb{R}^k} =\left| F_P(y)-F_P(z)\right|_{\mathbb{R}^k} \le \left|F_P^n(y)-F_P^n(z)\right|_{\mathbb{R}^n} +\Psi \dist(y, z),\quad \forall y, z \in B.
\end{equation}
On the other hand (\ref{eq:splitsplit}) yields
\begin{equation}
\intav_{4B}\left| |\nabla (v \cdot F_P^n)|^2-1 \right|\di \meas<\Psi, \quad \forall v \in \mathbb{S}^{n-1}.
\end{equation}
Therefore applying Lemma \ref{lem:sharpharm} again, we have
\begin{equation}\label{eq}
\left| v \cdot (F^n_P(y)- F^n_P(z))\right|_{\mathbb{R}^n} \le (1+\Psi)\dist (y, z),\quad \forall v \in \mathbb{S}^{n-1},\,\,\forall y, z \in B.
\end{equation}
Taking the supremum with respect to $v \in \mathbb{S}^{n-1}$ in (\ref{eq}) shows that $F_P^n|_B$ is $(1+\Psi)$-Lipschitz. Combining this with (\ref{eq:1pmlip}), we conclude.
\end{proof}
\begin{remark}
Although we checked $l=n$ in the proof of Lemma \ref{lem:sharplipbd}, the coincidence $l=n$ does not play an essential role in the proof. However this will play a role in the proof of Proposition \ref{prop:bihol} when we apply Proposition \ref{prop:lowerholderndim}.
\end{remark}
\begin{corollary}\label{corsharplip}
If a regular map $F=(f_1, \ldots, f_k):4rB \to \mathbb{R}^k$ satisfies
\begin{equation}
\intav_{4rB}\left| g-F^*g_{\mathbb{R}^k}\right| \di \meas +r^2\|\nabla \Delta f_i\|_{L^{\infty}}<\epsilon, \quad \|\nabla f_i\|_{L^{\infty}} \le L,\quad \forall i,
\end{equation}
then 
\begin{equation}\label{eq:locallipsharp}
\left| F(y)-F(z)\right|_{\mathbb{R}^k}\le (1+\Psi(\epsilon, r| K, N, L, k))\dist (y, z),\quad \forall y, z \in rB.
\end{equation}
\end{corollary}
\begin{proof}
Let us consider the rescaled objects;
\begin{equation}\label{eq:resc}
(\tilde{X}, \tilde{\dist}, \tilde{\meas}, \tilde{x}):=\left(X, \frac{1}{r}\dist, \frac{1}{\meas (rB)}\meas, x\right), \quad \tilde{F}:=\frac{1}{r}\cdot F=(\tilde{f}_{1}, \ldots, \tilde{f}_{k}).
\end{equation}
As in the proof of Proposition \ref{thm:transformation} we will use again similar notations, $\tilde{\nabla}, \tilde{\Delta}$ etc.
Note that $(\tilde{X}, \tilde{\dist}, \tilde{\meas})$ is an $\RCD(r^2K, N)$ space with
\begin{equation}
|\tilde{\nabla}\tilde{f}_{i}|=|\nabla f_i| \le L, \quad |\tilde{\nabla}\tilde{\Delta}\tilde{f}_{i}|=r^2|\nabla \Delta f_i| <\epsilon, 
\end{equation}
and
\begin{equation}
\intav_{4\tilde{B}}\left| \tilde{g}-\tilde{F}^*g_{\mathbb{R}^k}\right|\di \tilde{\meas}=\intav_{4rB}\left| g-F^*g_{\mathbb{R}^k}\right| \di \meas <\epsilon.
\end{equation}
Thus we can apply Lemma \ref{lem:sharplipbd} for (\ref{eq:resc}) to get (\ref{eq:locallipsharp}).
\end{proof}
\subsection{H\"older lower bound}
The main purpose of this subsection is to establish a H\"older lower bound for a regular map $F$:
\begin{equation}
(1-\epsilon)\dist(y, z)^{1+\epsilon} \le |F(y)-F(z)|_{\mathbb{R}^k},\quad \forall y, z \in B
\end{equation}
under assuming that the average of $|g-F^*g_{\mathbb{R}^k}|$ is small.
First let us consider a regular splitting map $F:4B \to \mathbb{R}^n$.
\begin{proposition}\label{prop:lowerholderndim}
If  $(X, \dist, \meas)$ is an $\RCD(-\epsilon, N)$ space with 
\begin{equation}\label{eq:reifflat0}
\dist_{\mathrm{pmGH}}\left( \left(X, \frac{1}{t}\dist, \frac{1}{\meas (tB)}\meas, y\right), \left(\mathbb{R}^n, \dist_{\mathbb{R}^n}, \frac{1}{\omega_n}\mathcal{H}^n, 0_n\right)\right)<\epsilon,\quad \forall y \in 3B,\,\,\forall t \in (0, 1],
\end{equation}
and a regular map $F=(f_1, \ldots, f_n):4B \to \mathbb{R}^n$ satisfies
\begin{equation}\label{eq:99ii}
\intav_{4B}\left| \langle \nabla f_i, \nabla f_j\rangle -\delta_{ij}\right|\di \meas + \|\Delta f_i\|_{L^{\infty}} <\epsilon, \quad \forall i,j
\end{equation}
then $F|_B$ gives a $\Psi$-Gromov-Hausdorff approximation to the image with
\begin{equation}\label{eq:almostisom}
 (1-\Psi) B_{\mathbb{R}^n} \subset F(B) \subset (1+\Psi) B_{\mathbb{R}^n}
\end{equation}
and
\begin{equation}\label{eq:concl}
(1-\Psi)\dist (y, z)^{1+\Psi} \le \left|F(y)-F(z)\right|_{\mathbb{R}^n} \le C(N)\dist (y, z),\quad \forall y, z \in B
\end{equation}
where $B_{\mathbb{R}^n}=B(F(x), 1)$ and $\Psi=\Psi(\epsilon |N)$.
\end{proposition}
\begin{proof}
It follows from (\ref{eq:reifflat0}) and (\ref{eq:99ii}) (with \cite[Theorem 4.4]{AmbrosioHonda2}) that $F|_B$ gives a $\Psi$-Gromov-Hausdorff approximation to the image which is $\Psi$-Hausdorff close to $B_{\mathbb{R}^n}$ (see \cite[Proposition 1.5]{BPS} or a proof of Proposition \ref{prop:splittingexample}).
Moreover the $C(N)$-Lipschitz continuity of $F$ stated in (\ref{eq:concl}) directly comes from Corollary \ref{lem:lipcont} with (\ref{eq:reifflat0}).

In order to prove remaining statements, fix $y, z \in B$ with $y \neq z$. Then  (\ref{eq:99ii}) implies
\begin{equation}
\intav_{B(y, 3)}\left|\langle \nabla f_i, \nabla f_j\rangle -\delta_{ij}\right|\di \meas<C(N)\epsilon, \quad \forall i, j.
\end{equation}
Letting $\rho=\dist(y, z)/2 \in (0, 1)$ and
applying Proposition \ref{thm:transformation} with Corollary \ref{cor:trans} show
\begin{equation}
\intav_{B(y, 3\rho)}\left| \langle \nabla f_{B(y, \rho), i}, \nabla f_{B(y, \rho), j}\rangle-\delta_{ij}\right|\di \meas \le \Psi,\quad \forall i,j
\end{equation}
with
\begin{equation}
\rho |\Delta f_{B(y, \rho), i}| \le C\rho^{1-\Psi},\quad \forall i
\end{equation} 
and
\begin{equation}
\left| T_{F, B(y, \rho)}\right|_{\infty} \le \rho^{-\Psi}
\end{equation}
which in particular implies that the smallest eigenvalue of $T_{F, B(y, \rho)}^{-1}$ is at least $\rho^{\Psi}$, where $\Psi=\Psi(\epsilon |N)$. Therefore
\begin{equation}\label{eq:helderest}
\left| F(y)-F(z)\right|_{\mathbb{R}^n}=\left|  \left(F_{B(y, \rho)}(y)-F_{B(y, \rho)}(z)\right) \cdot (T_{F, B(y, \rho)}^t)^{-1}\right|_{\mathbb{R}^n} \ge \rho^{\Psi} \left| F_{B(y, \rho)}(y)-F_{B(y, \rho)}(z)\right|_{\mathbb{R}^n}.
\end{equation}
On the other hand, consider the rescaled objects (with the same notations as in the proof of Corollary \ref{corsharplip});
\begin{equation}
(\tilde{X}, \tilde{\dist}, \tilde{\meas}, \tilde{y}):=\left(X, \frac{1}{\rho}\dist, \frac{1}{\meas (B(y, \rho))}\meas, y\right), \quad \tilde{F}:=\frac{1}{\rho}\cdot F_{B(y, \rho)}=(\tilde{f}_{1}, \ldots, \tilde{f}_{k}).
\end{equation}
Since $(\tilde{X}, \tilde{\dist}, \tilde{\meas})$ is an $\RCD(-\epsilon, N)$ space and 
\begin{equation}
\intav_{\tilde{B}(\tilde{y}, 3)}\left| \langle \tilde{\nabla}\tilde{f}_i, \tilde{\nabla}\tilde{f}_{j}\rangle -\delta_{ij}\right|\di \tilde{\meas} + \|\tilde{\Delta}\tilde{f}_{i}\|_{L^{\infty}}<\epsilon,\quad \forall i, j,
\end{equation}
it holds that $\tilde{F}|_{\tilde{B}(y, 2)}$ gives a $\Psi$-GH-approximation to the image. Therefore
\begin{equation}
\left|\tilde{F}(y)-\tilde{F}(z)\right|_{\mathbb{R}^n} \ge 2(1-\Psi),
\end{equation}
in other words,
\begin{equation}\label{eqtransholder}
\left|F_{B(y, \rho)}(y)-F_{B(y, \rho)}(z)\right|_{\mathbb{R}^n} \ge (1-\Psi)\dist (y, z).
\end{equation}
Thus we get (\ref{eq:concl}) because of (\ref{eq:helderest}) and (\ref{eqtransholder}). Finally (\ref{eq:almostisom}) comes directly from invariance of domain (see also \cite[Theorem 2.9 and Remark 2.10]{KM}). 
\end{proof}
\begin{corollary}\label{thm:canon_Reif}
Assume
\begin{equation}\label{eq:reifflat}
\dist_{\mathrm{pmGH}}\left( \left(X, \frac{1}{t}\dist, \frac{1}{\meas (B(y, t))}\meas, y\right), \left(\mathbb{R}^n, \dist_{\mathbb{R}^n}, \frac{1}{\omega_n}\mathcal{H}^n, 0_n\right)\right)<\epsilon,\quad \forall y \in 3rB,\,\,\forall t \in (0, r].
\end{equation}
If a regular map $F=(f_1,\ldots, f_n):4rB \to \mathbb{R}^n$ satisfies 
\begin{equation}\label{eq:split}
\intav_{4rB}\left|\langle \nabla f_i, \nabla f_j\rangle -\delta_{ij}\right|\di \meas +r\|\Delta f_i\|_{L^{\infty}} <\epsilon, \quad \forall i, j, 
\end{equation}
then $F|_{rB}$ gives a $\Psi r$-Gromov-Hausdorff approximation to the image with 
\begin{equation}
(1-\Psi)rB_{\mathbb{R}^n} \subset F(rB) \subset (1+\Psi)rB_{\mathbb{R}^n} 
\end{equation}
and
\begin{equation}\label{eq:holder}
(1-\Psi)\dist (y, z)^{1+\Psi}\le \left| F(y)-F(z)\right|_{\mathbb{R}^n}\le C(N)\dist (y, z),\quad \forall y, z \in rB,
\end{equation}
where $\Psi=\Psi(\epsilon, r| K, N)$ and $B_{\mathbb{R}^n}=B(F(x), 1)$. 
\end{corollary}
\begin{proof}
The conclusion follows from applying Proposition \ref{prop:lowerholderndim} for the rescaled objects;
\begin{equation}
\left(X, \frac{1}{r}\dist, \frac{1}{\meas (rB)}\meas, x\right), \quad F_r:=\frac{1}{r}\cdot F
\end{equation}
because the rescaled space is an $\RCD (r^2K, N)$ space, thus 
the conclusions hold for $\Psi= \Psi(\max\{\epsilon, r^2|K|\} |K, N)=\Psi(\epsilon, r|K, N)$.
\end{proof}
Next let us discuss a general regular map $F:4B \to \mathbb{R}^k$.
\begin{proposition}\label{prop:bihol}
If $(X, \dist, \meas)$ is an $\RCD(-\epsilon, N)$ space with (\ref{eq:reifflat0})
and a regular map $F=(f_1, \ldots, f_k):4B \to \mathbb{R}^k$ satisfies
\begin{equation}
\intav_{4B}\left|g-F^*g_{\mathbb{R}^k}\right|\di \meas+\|\Delta f_i\|_{L^{\infty}}<\epsilon,\quad \forall i,
\end{equation}
then
\begin{equation}
(1-\Psi)\dist(y, z)^{1+\Psi} \le \left| F(y)-F(z)\right|_{\mathbb{R}^k}\le C(N, k)\dist (y, z),\quad \forall y, z \in B,
\end{equation}
where $\Psi=\Psi(\epsilon| N, k)$.
\end{proposition}
\begin{proof}
By an argument similar to the proof of Lemma \ref{lem:sharplipbd} there exists $P \in O(k)$ such that the map $F_P=FP^t=(f_{P, 1},\ldots, f_{P,k}):4B \to \mathbb{R}^k$ satisfies
\begin{equation}
\intav_{4B}\left|\langle \nabla f_{P, i}, \nabla f_{P, j}\rangle -\delta_{ij}\right|\di \meas <\Psi, \quad \forall i, j \in \mathbb{N}_{\le n}.
\end{equation} 
Thus we can apply Proposition \ref{prop:lowerholderndim} for the map $F_P^n:=(f_{P, 1},\ldots, f_{P, n})$ to get;
\begin{equation}
(1-\Psi)\dist (y, z)^{1+\Psi}\le \left|F_P^n(y)-F_P^n(z)\right|_{\mathbb{R}^n} \le \left|F_P(y)-F_P(z)\right|_{\mathbb{R}^k} = \left|F(y)-F(z)\right|_{\mathbb{R}^k}
\end{equation}
for all $y, z \in B$. The desired $C(N, k)$-Lipschitz continuity of $F$ is a direct consequence of Corollary \ref{lem:lipcont} with (\ref{eq:reifflat0}).
\end{proof}
As in the proof of Corollary \ref{thm:canon_Reif}, a rescaling argument with Proposition \ref{prop:bihol} allows us to prove the following, where we omit the proof.
\begin{corollary}\label{cor:reifen}
If (\ref{eq:reifflat}) holds and a regular map $F=(f_1, \ldots, f_k):4rB \to \mathbb{R}^k$ satisfies
\begin{equation}\label{eq:almostiso}
\intav_{4rB}\left| g-F^*g_{\mathbb{R}^k}\right|\di \meas + r\left\|\Delta f_i\right\|_{L^{\infty}} <\epsilon, \quad \forall i,
\end{equation}
then
\begin{equation}\label{eq:holderreifenberg}
(1-\Psi)\dist (y, z)^{1+\Psi} \le \left|F(y)-F(z)\right|\le C(N, k)\dist (y, z)\quad \forall y, z \in rB,
\end{equation}
where $\Psi=\Psi(\epsilon, r|K, N, k)$.
\end{corollary}
\section{Proof of Theorem \ref{thm:canonicaldiff}}\label{section5}
We are now in a position to prove Theorem \ref{thm:canonicaldiff}.
\begin{proof}[Proof of (\ref{1}) of Theorem \ref{thm:canonicaldiff}]
Let us recall statements we want to prove; if $F_i:X_i \to M^N$ is equi-regular with
\begin{equation}\label{eeeqqrrcc}
\intav_{X_i}\left|g_i-F_i^*g_{M^N}\right| \di \meas_i\to 0,
\end{equation}
then;
\begin{enumerate}[I]
\item\label{home} the map $F_i$ gives an $\epsilon_i$-Gromov-Hausdorff approximation for some $\epsilon_i \to 0^+$;
\item\label{bihole} for any $\epsilon \in (0, 1)$ there exists $i_0 \in \mathbb{N}$ such that for any $i \ge i_0$, $F_i$ gives a homeomorphism with
\begin{equation}\label{eq:global}
(1-\epsilon)\dist_i(x, y)^{1+\epsilon}\le \dist_g(F_i(x), F_i(y)) \le C(N)\dist_i(x, y),\quad \forall x, y \in X_i.
\end{equation}
Moreover if each $(X_i, \dist_i)$ is isometric to an $N$-dimensional closed Riemannian manifold $(M^N_i, \dist_{g_i})$ and $F_i$ is smooth for any $i$, then $F_i$ is diffeomorphism for any sufficiently large $i$. 
\end{enumerate}

First let us prove (\ref{home}).
It is enough to check that $F_i$ has a convergent subsequence to an isometry $F:M^N \to M^N$ with respect to (\ref{eq:mghconv}).
Without loss of generality we can assume that $F_i$ converge to a Lipschitz map $F:M^N \to M^N$ because of Corollary \ref{cor:reifen}.

Let $\Phi=(\phi_1, \ldots, \phi_k):M^N \hookrightarrow \mathbb{R}^k$ be a smooth isometric embedding, where $k$ depends only on $N$.
Denoting by $\Phi \circ F=(f_1,\ldots, f_k)$, the equi-regularity for $F_i$ implies $\Delta f_i \in L^{\infty}$ for any $i$. Thus (for instance by the $L^p$-Calder\'on-Zygmund inequality (\ref{eq:cz}) on $M^N$) $f_i \in C^1(M^N)$. Moreover (\ref{eeeqqrrcc}) yields $g=F^*g$.
In particular $F$ is a local homeomorphism and $1$-Lipschitz. Thus invariance of domain shows that $F(M^N)$ is open, thus $F(M^N)=M^N$ because the compactness of $M^N$ implies that $F(M^N)$ is closed. Recalling that any surjective $1$-Lipschitz map from a compact metric space to itself is actually an isometry, $F$ is an isometry. Thus we have  (\ref{home}).

Next let us prove (\ref{bihole}).
Fix $\epsilon \in (0, 1)$. It follows from Theorem \ref{thm:unifreifen} that there exist $i_0 \in \mathbb{N}$ and $r \in (0, 1)$ such that 
\begin{equation}\label{eq:reifflat1}
\dist_{\mathrm{pmGH}}\left( \left(X_i, \frac{1}{t}\dist_i, \frac{1}{\meas_i (B(x, t))}\meas_i, x\right), \left(\mathbb{R}^N, \dist_{\mathbb{R}^N}, \frac{1}{\omega_N}\mathcal{H}^N, 0_N\right)\right)<\epsilon,\quad \forall i \ge i_0,\,\,\forall x \in X_i,\,\,\forall t \in (0, r].
\end{equation}
and
\begin{equation}
\intav_{B(x, 4r)}\left| g_i-F_i^*g\right|\di \meas_i <\epsilon,\quad \forall i \ge i_0,\,\,\forall x \in X_i.
\end{equation}
Thus Corollary \ref{cor:reifen} shows  
\begin{equation}\label{eq:global2}
(1-\Psi)\dist_i(x, y)^{1+\Psi}\le \dist_g(F_i(x), F_i(y)) \le C(N)\dist_i(x, y), \quad \forall i \ge i_0,\,\,\forall x, y \in X_i \text{ with $\dist_i(x, y) \le r$},
\end{equation}
where $\Psi=\Psi(\epsilon, r|K, N, L)$.
By (\ref{home}), there exists $i_1 \in \mathbb{N}_{\ge i_0}$ such that 
\begin{equation}\label{eq:global3}
(1-\Psi)\dist_i(x, y)^{1+\Psi}\le \dist_g(F_i(x), F_i(y)) \le C(N)\dist_i(x, y),\quad \forall i \ge i_1,\,\,\forall x, y \in X_i.
\end{equation}
In particular $F_i$ is injective and a local homeomorphism.  Since $X_i$ is a topological $N$-dimensional manifold, invariance of domain yields that $F_i$ is a homeomorphism.

Finally let us prove the remaining statement for the diffeomorphism, where we just follow the same argument as discussed in \cite[Proposition 4.2]{WangZhao}. Assume that each  $(X_i, \dist_i)$ is isometric to an $N$-dimensional closed Riemannian manifold and that $F_i$ is smooth. The proof is done by a contradiction. If $F_i$ is not a diffeomorphism for any sufficiently large $i$, then for such an $i$, there exists $x \in X_i$ such that the differential
\begin{equation}
DF_i|_x: T_xX_i \to T_{F_i(x)}M^N
\end{equation}
is not an isomorphism. Then by Taylor's theorem it is easy to find a positive constant $C \in (0, \infty)$ and a sequence $y_j \to x$ in $X_i$ with $y_j \neq x$ and
\begin{equation}
\dist_g(F_i(x), F_i(y_j))\le C\dist_i(x, y_j)^2,\quad \forall j.
\end{equation}
Thus
\begin{equation}
\frac{1}{2} \cdot \dist_i(x, y_j)^{3/2} \le C\dist_i(x, y_j)^2,
\end{equation}
namely
\begin{equation}\label{eq:2order}
0<\frac{1}{2} \le C\dist_i(x, y_j)^{1/2} \to 0^+,\quad (j \to \infty),
\end{equation}
which is a contradiction.
\end{proof}
\begin{proof}[Proof of (\ref{2}) of Theorem \ref{thm:canonicaldiff}]
Under the assumption of (\ref{2}) (namely, fixing a uniform convergence of equi-regular maps $\Phi_i$ to a smooth isometric embedding $\Phi$), we know
\begin{equation}\label{1000}
\int_{X_i}\left|g_i-\Phi_i^*g_{\mathbb{R}^k}\right|\di \meas_i \to \int_{M^N}\left|g-\Phi^*g_{\mathbb{R}^k}\right|\di \mathcal{H}^K=0.
\end{equation}
Fix $\epsilon \in (0, 1)$.
Then by an argument similar to the proof of (\ref{eq:global2}), Corollary \ref{corsharplip} with (\ref{1000}) shows that there exist $r \in (0, 1)$ and $i_0 \in \mathbb{N}$ such that 
\begin{equation}
|\Phi_i(x)-\Phi_i(y)|_{\mathbb{R}^k}\le (1+\epsilon)\dist_i (x, y),\quad \forall i \ge i_0,\,\,\forall x, y \in X_i\,\,\text{with $\dist_i(x, y)\le r$.}
\end{equation}
On the other hand the $C^{1, 1}$-regularity of $\pi$ with (\ref{eq:leipnit}) and (\ref{1000}) shows that $F_i$ is equi-regular with
\begin{equation}
\intav_{X_i}\left| g_i-F_i^*g\right|\di \meas_i \to 0.
\end{equation}
Moreover since $\pi$ is locally $(1+\epsilon)$-Lipschitz on a neighborhood of $\Phi (M^N)$ because $\pi$ is a Riemannian submersion, there exists $i_1 \in \mathbb{N}_{\ge i_0}$ such that 
\begin{equation}\label{eq:2000}
\dist_g(F_i(x), F_i(y))\le (1+\epsilon)^2 \dist_i (x, y),\quad \forall i \ge i_1,\,\,\forall x, y \in X_i\,\,\text{with $\dist_i(x, y)\le r$}.
\end{equation}
Thus since $F_i$ gives an $\epsilon_i$-Gromov-Hausdorff approximation by (\ref{1}), it follows from (\ref{eq:2000}) that there exists $i_2 \in \mathbb{N}_{\ge i_1}$ such that 
\begin{equation}\label{eq:20001}
\dist_{g}(F_i(x), F_i(y))\le (1+\epsilon)^3 \dist_i (x, y),\quad \forall i \ge i_2,\,\,\forall x, y \in X_i,
\end{equation}
which proves (\ref{2}).
\end{proof}
\begin{proof}[Proof of (\ref{3}) of Theorem \ref{thm:canonicaldiff}]
Let us follow the same arguments in \cite[Theorem 1.1]{H18}.
Let $\psi_j=\Delta \phi_j \in C^{\infty}(M^N)$ (recall that we fix a smooth isometric embedding $\Phi=(\phi_1,\ldots, \phi_k):M^N \hookrightarrow \mathbb{R}^k$). Without loss of generality we can assume
\begin{equation}
\int_{M^N}\phi_i\di \mathcal{H}^N=0.
\end{equation}
Then applying \cite[Lemma 2.10]{AmbrosioHonda2}, there exists a sequence of equi-Lipschitz functions $\psi_{i,j}:X_i \to\mathbb{R}$ such that $\psi_{i,j}$ $H^{1,2}$-strongly converge to $\psi_j$ with
\begin{equation}
\int_{X_i}\psi_{j, i}\di \meas_i=0.
\end{equation}
Therefore by the Poincar\'e-Sobolev inequality (c.f. (\ref{eq1068})) there exists a unique $\phi_{i, j} \in D(\Delta_i)$ such that $\Delta_i\phi_{i, j}=\psi_{i, j}$ with
\begin{equation}
\int_{X_i}\phi_{i, j}\di \meas_i=0.
\end{equation}
Applying \cite[Theorem 4.4]{AmbrosioHonda2},
after passing to a subsequence with no loss of generality we can assume that $\phi_{i, j}$ $H^{1,2}$-strongly converge to some $\tilde{\phi}_j \in D(\Delta)$ with $\Delta \tilde{\phi}_j=\psi_j$. Since 
\begin{equation}
\int_{M^N}\tilde{\phi}_j\di \mathcal{H}^N=\lim_{i \to \infty}\int_{X_i}\phi_{i, j}\di \meas_i=0,
\end{equation}
we have $\phi_j=\tilde{\phi}_j$. In particular the sequence of maps $\Phi_i=(\phi_{i, 1},\ldots, \phi_{i, k}):X_i \to \mathbb{R}^k$ converge uniformly to $\Phi$ and is equi-regular with
\begin{equation}\label{1001}
\sup_{i, j}\|\nabla_i\Delta_i \phi_{i, j}\|_{L^{\infty}}<\infty, \quad \int_{X_i}\left|g_i-\Phi_i^*g_{\mathbb{R}^k}\right|\di \meas_i \to 0,
\end{equation}
which proves (\ref{a}).
Moreover we also have (\ref{b}) because we can choose $\psi_{i, j}$ as a smooth function when $(X_i, \dist_i)$ is isometric to a smooth closed Riemannian manifold, with the elliptic regularity theorem.
\end{proof}
Finally in order to show that Theorem \ref{thm:canonicaldiff} is sharp (see Remark \ref{rem:sharp}), let us prepare the following.
\begin{proposition}\label{propbilip}
Let $(X, \dist, \mathcal{H}^N, x)$ be a pointed non-collapsed $\RCD(K, N)$ space and let $F:B(x, 1) \to M^k$ be a regular map into an $k$-dimensional (not necessarily complete) Riemannian manifold $(M^k, g_{M^k})$.
If $F$ is bi-Lipschitz, then $B(x, 1) \subset \mathcal{R}_N$.
\end{proposition}
\begin{proof}
It is enough to prove $x \in \mathcal{R}_N$. Fix a sequence $r_i \to 0^+$ and a smooth isometric embedding $\Phi:M^k \hookrightarrow \mathbb{R}^l$. After passing to a subsequence there exist a pointed non-collapsed $\RCD(0, N)$ space $(Z, \dist_Z, \meas_Z, p)$ and a harmonic map $\tilde{\Phi}:Z \to \mathbb{R}^l$ such that 
\begin{equation}\label{eq:tangen}
\left(X, \frac{1}{r_i}\dist, \mathcal{H}^N, x\right) \stackrel{\mathrm{pmGH}}{\to} (Z, \dist_Z, \mathcal{H}^N, p)
\end{equation}
and that $r^{-1}_i(\Phi \circ F-\Phi \circ F(x))$ locally uniformly converge to $\tilde{\Phi}$ with respect to (\ref{eq:tangen}). In particular $\tilde{\Phi}$ is a bi-Lipschitz embedding. Thus applying \cite[Corollary 4.10]{HS} yields that $(Z, \dist_Z, \mathcal{H}^N, p)$ is isometric to $(\mathbb{R}^N,\dist_{\mathbb{R}^N}, \mathcal{H}^N, 0_N)$. Therefore we conclude.
\end{proof}
\section{Canonical topological sphere and torus theorems}\label{section6}
As applications of Theorem \ref{thm:canonicaldiff}, let us introduce new canonical topological stability results.
First let us give a proof of  Theorem \ref{thm:hemisphere}. By a sequential compactness of $\RCD(K, N)$ spaces with respect to the measured Gromov-Hausdorff convergence,  it is enough to check the following. 
\begin{theorem}\label{thm:spherecanonical}
Let $K \in \mathbb{R}, N \in \mathbb{N}$ and let 
\begin{equation}
(X_i, \dist_i, \meas_i) \stackrel{\mathrm{mGH}}{\to} (\mathbb{S}^N, \dist_{\mathbb{S}^N}, \mathcal{H}^N)
\end{equation}
be a measured Gromov-Hausdorff convergent sequence of compact $\RCD(K, N)$ spaces. Then for any sufficiently large $i$, the map $F_i:X_i \to \mathbb{S}^N$;
\begin{equation}
F_i:=\left(\sum_{j=1}^{N+1}f_{i, j}^2\right)^{-1/2}\cdot \left(f_{i, 1},\ldots, f_{i, N+1}\right), \quad \text{where}\,\,\,\intav_{X_i}f_{i, j}^2\di \meas_i =\frac{1}{N+1}, \quad \forall j,
\end{equation}
gives an well-defined homeomorphism and an $\epsilon_i$-Gromov-Hausdorff approximation with
\begin{equation}\label{eq:bihlip}
(1-\epsilon_i)\dist_i (x, y)^{1+\epsilon_i}\le \dist_{\mathbb{S}^N}(F_i(x), F_i(y)) \le (1+\epsilon_i) \dist_i(x, y),\quad \forall x, y \in X_i
\end{equation}
for some $\epsilon_i \to 0^+$, where $f_{i, j}$ is an eigenfunction of $-\Delta_i$ with the $j$-th eingenvalue $\lambda_{i, j}$.
\end{theorem}
\begin{proof}
Thanks to the spectral convergence result \cite[Theorem 7.8]{GigliMondinoSavare13}, the map $\Phi_i=(f_{i, 1},\ldots, f_{i, N+1}):X_i \to \mathbb{R}^{N+1}$ satisfies that 
\begin{equation}
\intav_{X_i}\left|g_i-\Phi_i^*g_{\mathbb{R}^{N+1}}\right|\di \meas_i \to 0, \quad \sup_{i, j}\left(\|\nabla_i\Delta_if_{i, j}\|_{L^{\infty}}+\|\Delta_if_{i, j}\|_{L^{\infty}}\right)<\infty,
\end{equation}
and that after passing to a subsequence, $\Phi_i$ converge uniformly to the canonical inclusion $\mathbb{S}^n \hookrightarrow \mathbb{R}^{N+1}$ up to multiplying a $P \in O(N+1)$. Thus applying Theorem \ref{thm:canonicaldiff} for $\Phi_i$ and the canonical map
$
\pi_{\mathbb{R}^{N+1}}:\mathbb{R}^{N+1}\setminus \{0_{N+1}\} \to \mathbb{S}^{N}
$
defined by
\begin{equation}\label{eq:projecdef}
\pi_{\mathbb{R}^{N+1}}(x):=\frac{1}{|x|_{\mathbb{R}^{N+1}}}\cdot x,
\end{equation}
we conclude.
\end{proof}
\begin{remark}\label{rem:sharp}
Let us provide a simple example which shows that $F_i$ stated in Theorem \ref{thm:spherecanonical} cannot be improved to bi-Lipschitz homeomorphisms. For all $r \in (0, 1)$ and $N \in \mathbb{N}_{\ge 2}$, consider the $(N-1)$-spherical suspension $(\mathbb{S}^0*\mathbb{S}^{N-1}(r), \dist_{\mathbb{S}^0*\mathbb{S}^{N-1}(r)}, \mathcal{H}^N)$ of $(\mathbb{S}^{N-1}(r), \dist_{\mathbb{S}^{N-1}(r)}, \mathcal{H}^{N-1})$,
where
$\mathbb{S}^{N-1}(r):=\{x \in \mathbb{R}^N; |x|_{\mathbb{R}^N}=r\}$ with the standard intrinsic distance $\dist_{\mathbb{S}^{N-1}(r)}$. 
Note that $(\mathbb{S}^0*\mathbb{S}^{N-1}(r), \dist_{\mathbb{S}^0*\mathbb{S}^{N-1}(r)}, \mathcal{H}^N)$ is an $\RCD(N-1, N)$ space, has two singular points and measured Gromov-Hausdorff converge to $(\mathbb{S}^N, \dist_{\mathbb{S}^N}, \mathcal{H}^N)$ as $r \to 1^-$. Denote by $F_r:\mathbb{S}^0*\mathbb{S}^{N-1}(r) \to \mathbb{S}^N$ the corresponding map obtained in Theorem \ref{thm:spherecanonical}.
Then Proposition \ref{propbilip} tells us that $F_r$ is not bi-Lipschitz.
\end{remark}
Finally let us provide a canonical topological torus theorem. Since it is a direct consequence of Theorem \ref{thm:canonicaldiff}  as in Theorem \ref{thm:spherecanonical}, we omit the proof. See \cite{GR, MMP, MW} for related results.
\begin{theorem}\label{thm:flat}
Let $K \in \mathbb{R}, N \in \mathbb{N}$ and let 
\begin{equation}
(X_i, \dist_i, \meas_i) \stackrel{\mathrm{mGH}}{\to} (\mathbb{S}^1(r_1)\times \cdots \times \mathbb{S}^1(r_N), \dist, \mathcal{H}^N)
\end{equation}
be a measured Gromov-Hausdorff convergent sequence of compact $\RCD(K, N)$ spaces for some $r_1\ge r_2 \ge \cdots \ge r_N>0$, where $\dist$ denotes the canonical flat intrinsic distance on $\mathbb{S}^1(r_1)\times \cdots \times \mathbb{S}^1(r_N)$. Then for any sufficiently large $i$, the map $F_i:X_i \to \mathbb{S}^1(r_1)\times \cdots \times \mathbb{S}^1(r_N)$;
\begin{equation}
F_i(x):=\left( \pi_{\mathbb{R}^2}\left(f_{i, 1}(x), f_{i, 2}(x)\right), \ldots, \pi_{\mathbb{R}^2}\left(f_{i, 2N-1}(x), f_{i, 2N}(x)\right)\right)
\end{equation}
gives an well-defined homeomorphism (see (\ref{eq:projecdef}) for the definition of $\pi_{\mathbb{R}^2}$) and an $\epsilon_i$-Gromov-Hausdorff approximation with
\begin{equation}\label{eq:bihlip3}
(1-\epsilon_i)\dist_i (x, y)^{1+\epsilon_i}\le \dist (F_i(x), F_i(y)) \le (1+\epsilon_i) \dist_i(x, y),\quad \forall x, y \in X_i
\end{equation}
for some $\epsilon_i \to 0^+$,
where $f_{i, j}$ denotes an eigenfunction of $-\Delta_i$ on $(X_i, \dist_i, \meas_i)$ with the $j$-th eigenvalue $\lambda_{i, j}$ and
\begin{equation}
\intav_{X_i}f_{i, 2j-1}^2\di \meas_i=\intav_{X_i}f_{i, 2j}^2\di \meas_i=\frac{r_j^2}{2},\quad \forall i, j.
\end{equation}
\end{theorem}
\section{Appendix; a maximum principle}\label{section7}
Let us emphasize that most techniques in \cite{WangZhao} are also available for $\RCD(n-1, n)$ spaces. In particular Theorem \ref{thm:spherecanonical} can be also proved along the same way if we establish a maximum principle which generalizes \cite[Theorem 7.1]{Peterson} to $\RCD(n-1, n)$ spaces. Let us provide the proof for any compact $\RCD(K, N)$ spaces as an independent interest, which relies on the standard Moser iteration technique.
\begin{proposition}\label{lem:subharm2}
Let $(X,\dist,\meas)$ be a compact $\RCD(K,N)$ space with the diameter at most $d$ and $\meas(X)=1$, and let $u\in H^{1,2}(X, \dist, \meas)$. Assume that there exists $f\in L^p(X, \meas)$ for some $p>\frac N 2$ such that $f(x) \ge 0$ for $\meas$-a.e. $x \in X$ and that 
\begin{equation}
\int_X\<\nabla u,\nabla \phi\>\ \d\m\le \int_X f\phi\ \d\m, \quad \forall \phi \in H^{1, 2}(X, \dist, \meas)\,\,\text{with $\phi(x) \ge 0$ for $\meas$-a.e. $x \in X$.}
\label{eq1061}
\end{equation}
Then there exists $C=C(K, N, p, d)$ such that
\begin{equation}
\esssup{X}\ u-\int_X u\ \d\m\le C\|f\|_{L^p}^\frac 1 2\left(\left\|u-\int_X u\ \d\m\right\|_{L^\frac{p}{p-1}}^\frac 1 2+\|f\|_{L^p}^\frac 1 2\right).
\label{eq1060}
\end{equation}
\end{proposition}
\begin{proof}
For arbitrary $k>0$, define $w=w_k\in H^{1,2}(X, \dist, \meas)$ by 
\begin{equation}
w:=\left(u-\int_X u\ \d\m\right)_++k.
\label{eq1062}
\end{equation}
We also fix $A>k$ and $\beta\ge 1$. Define a function $H=H_{k, A, \beta}:[k,\infty)\ra [0,\infty)$ by
\begin{equation}
H(t):=\left\{
\begin{aligned}
&t^\beta-k^\beta, &\ & (t\le A)\\
&\beta A^{\beta-1}t+(1-\beta)A^\beta-k^\beta, &\ & (t>A)
\end{aligned}
\right.
\label{eq1063}
\end{equation}
and define a function $G:[k,\infty)\ra [0,\infty)$ by
\begin{equation}
G(t):=\int_k^t H'(s)^2\ \d s.
\label{eq1064}
\end{equation}
It is easily seen that $H$ is $\beta A^{\beta-1}$-Lipschitz, thus $G\circ w\in H^{1,2}(X, \dist, \meas)$ with
\begin{equation}
\int_X|\nabla H(w)|^2\ \d\m=\int_X G'(w)|\nabla w|^2\ \d\m=\int_X\langle \nabla u, \nabla G(w)\rangle \le \int_X fG(w)\ \d\m.
\label{eq1065}
\end{equation}
Since $G(t)\le tG'(t)$ for all $t\ge k$, we have 
\begin{equation}
\int_X|\nabla H(w)|^2\ \d\m\le \frac 1 k\int_X fw^2H'(w)^2\ \d\m.
\label{eq1066}
\end{equation}
On the other hand recall the Poincar\'e-Sobolev inequality for some $C=C(K, N, d)$;
\begin{equation}
\left(\int_X\left|v-\int_X v\ \d\m\right|^\frac{2N}{N-2}\ \d\m\right)^\frac{N-2}{2N}\le C\left(\int_X|\nabla v|^2\ \d\m\right)^\frac 1 2, \quad \forall v \in H^{1,2}(X, \dist, \meas),
\label{eq1068}
\end{equation}
which is a direct consequence of the Bishop-Gromov inequality (\ref{eq:bg}) and the Poincar\'e inequality (\ref{eq1020}) (see \cite[Theorem 5.1]{HK}).
Let $v=H(w)$ and plug in \eqref{eq1066},
\begin{equation}
\left(\int_X\left|H(w)-\int_X H(w)\ \d\m\right|^\frac{2N}{N-2}\ \d\m\right)^\frac{N-2}{N}\le \frac C k\int_X fw^2H'(w)^2\ \d\m.
\label{eq1069}
\end{equation}
By the H\"older inequality
\begin{equation}
\int_Xfw^2H'(w)^2\ \d\m\le \|f\|_{L^p}\left(\int_X \left(wH'(w)\right)^\frac{2p}{p-1}\ \d\m\right)^\frac{p-1}{p}.
\label{eq1070}
\end{equation}
Letting $A\ra \infty$ in (\ref{eq1069}) and (\ref{eq1070}), we have
\begin{equation}
\left(\int_X\left|w^\beta-\int_X w^\beta\ \d\m\right|^\frac{2N}{N-2}\ \d\m\right)^\frac{N-2}{2N}\le \frac {C\beta\|f\|^\frac 1 2_{L^p}}{k^\frac 1 2}\left(\int_X w^\frac{2\beta p}{p-1}\ \d\m\right)^\frac{p-1}{2p}.
\label{eq1071}
\end{equation}
By the triangle and the H\"older inequalities,
\begin{equation}
\left(\int_X w^\frac{2\beta N}{N-2}\ \d\m\right)^\frac{N-2}{2N}\le \left(\frac {C\beta\|f\|_{L^p}^\frac 1 2}{k^\frac 1 2}+
1\right)\left(\int_Xw^\frac{2\beta p}{p-1}\ \d\m\right)^\frac{p-1}{2p}.
\label{eq1072}
\end{equation}
Choosing $\beta:=N(p-1)/p(N-2)>1$ with iteration of \eqref{eq1072} yields
\begin{equation}
\esssup{X}\ w\le C'\|w\|_{L^\frac{2p}{p-1}},
\label{eq1073}
\end{equation}
where 
\begin{equation}
C':=\Prod{i=0}{\infty}\left(\frac {C\beta^i\|f\|_{L^p}^\frac 1 2}{k^\frac 1 2}+
1\right)^\frac 1{\beta^i}.
\label{eq1074}
\end{equation}
By the H\"older and the Young inequalities, 
\begin{equation}
\|w\|_{L^\frac{2p}{p-1}}\le t\esssup{X}\ w+t^\frac{1+p}{1-p}\|w\|_{L^1}. \quad \forall t>0.
\label{eq1075}
\end{equation}
Taking $t:=\frac 1{2C'}$, \eqref{eq1073} and \eqref{eq1075} imply
\begin{equation}
\esssup{X}\ w\le 2^\frac {2p}{p-1}{C'}^\frac{p+1}{p-1}\|w\|_{L^1}.
\label{eq1076}
\end{equation}
On the other hand
\begin{equation}
\|w\|_{L^1}=\left\|\left(u-\int_X u\ \d\m\right)_+\right\|_{L^1}+k\le C\left(\int_X\left|\nabla\left(u-\int_X u\ \d\m\right)_+\right|^2\ \d\m\right)^\frac 1 2+k.
\label{eq1077}
\end{equation}
By \eqref{eq1061} we have
\begin{align}
\int_X\left|\nabla\left(u-\int_X u\ \d\m\right)_+\right|^2\ \d\m&\le \int_X f \cdot \left(u-\int_X u\ \d\m\right)_+\ \d\m\nonumber \\
&\le \|f\|_{L^p}\left\|\left(u-\int_Xu\ \d\m\right)_+\right\|_{L^\frac{p}{p-1}}.
\label{eq1078}
\end{align}
Letting $k:=\|f\|_{L^p}$, then $C'$ is depending only on $K$, $N$ and $d$. Then by \eqref{eq1076}, \eqref{eq1077}, \eqref{eq1078}, we conclude.
\end{proof}

{

\end{document}